\documentclass[a4paper, 12pt]{article}
\usepackage[utf8]{inputenc}
\usepackage[top=2.5cm,bottom=3.0cm,left=2.5cm,right=2.5cm]{geometry}
\usepackage{amssymb,amsmath,amsfonts,amsthm}
\usepackage{graphics}
\usepackage{graphicx}  

\usepackage{color}
\usepackage{soul}

\usepackage{lipsum}

\usepackage[pagewise]{lineno} 

\providecommand{\keywords}[1]
{
	\small	
	\textbf{\textit{Keywords---}} #1
}

\newtheorem{theorem}{Theorem}[section]
\newtheorem{proposition}[theorem]{Proposition}
\newtheorem{lemma}[theorem]{Lemma}
\newtheorem{claim}[theorem]{Claim}
\newtheorem{corollary}[theorem]{Corollary}

\newtheorem{remark}{Remark}[section] 
\newtheorem{example}{Example}[section]

\newcommand{\oo}{\infty}

\newcommand{\comp}{\circ}

\newcommand{\N}{\mathbb{N}}
\newcommand{\Z}{\mathbb{Z}}
\newcommand{\C}{\mathbb{C}}
\newcommand{\R}{\mathbb{R}}

\bibdata{prob}
\bibstyle{alpha}

\newcommand{\e}{\epsilon}

\newcommand{\bm}{\mathbf}

\newcommand{\sub}{\subseteq}

\newcommand{\lav}{\left|}
\newcommand{\rav}{\right|}
\newcommand{\lf}{\left\lfloor}
\newcommand{\rf}{\right\rfloor}

\newcommand{\ldav}{\left| \left|}
\newcommand{\rdav}{\right| \right|}

\newcommand{\lc}{\left(} 
\newcommand{\rc}{\right)}
\newcommand{\lb}{\left[}
\newcommand{\rb}{\right]}
\newcommand{\lfp}{\left\{}
\newcommand{\rfp}{\right\}}
\newcommand{\mc}{\mathcal}

\newcommand{\SP}{\mathbb{S}}

\usepackage{caption}

\title{On Stochastic Adding Machines and Non-Autonomous Julia Sets}

\author{D. Caprio, 
	A. Messaoudi, 
	I. Tsokanos, 
	G. Valle}

\date{ }

\begin{document}
	
	\maketitle
	
	\begin{abstract}
		
	This work investigates the spectral and dynamical properties of transition operators associated with a class of stochastic adding machines. In particular, we characterize the spectra of these operators in terms of non-autonomous Julia sets. This description, in turn, provides new information about their dynamical behavior, including criteria for supercyclicity and hypercyclicity.

	\end{abstract}
	
	

\section{Introduction}\label{sec:intro}

Julia sets, introduced independently by Julia \cite{Julia} and Fatou \cite{fatou19}, are compact subsets of the complex plane that play a central role in complex dynamical systems. 
Given a non-constant complex polynomial $f: \mathbb{C} \to \C$, the \textit{filled Julia set $\mc{E}_{f}$} of $f$ is the set of complex numbers $z \in \C$  whose forward iteration, 
$$ \lfp f^{r}\lc z \rc \rfp_{r\ge 1}, \quad \text{where } f^{r}\lc z \rc= f \circ \ldots \circ f\lc z \rc \; \lc r\text{-times}\rc, $$
remains bounded. 
The \textit{Julia set} of $f$ is then defined as the boundary $\partial \mc{E}_{f}$ of $\mc{E}_{f}$. 
More generally, for a sequence of non-constant complex polynomials $\overline{f} = \lfp f_{r} \rfp_{r \ge 1}$, 
the \textit{non-autonomous filled Julia set} $\mc{E}_{\overline{f}}$ consists of those complex numbers $z$ whose forward orbit 
$$ \lfp \tilde{f}_{r}\lc z \rc \rfp_{r \ge 1}, \quad \text{with } \tilde{f}_{r} \; = \; f_{r} \circ f_{r-1} \circ \cdots \circ f_{1}, $$
is bounded. For further background on non-autonomous Julia sets, see \cite{Bruck-Geometric_Properties_of_Julia_Sets_of_the_Polynomials_of_the_form_z2+cn, Comerford-Non-autonomous_Julia_Sets_with_Measurable_Invariant_Sequences_of_Line_Fields, Sester-Hyperbolicite_des_Polynomes_Fibres}.

\paragraph{}

In 2000, Killeen and Taylor \cite{Killeen_Taylor-A_Stochastic Adding_Machine_And_Complex_Dynamics, Killeen_Taylor-How_The_Propagation_Of_Error_Through_Stochastic_Counters_Affects_Time_Discrimination_And_Other_Psychophysical_Judgments} uncovered a remarkable connection between stochastic adding machines and Julia sets. 
A stochastic adding machine is a Markov chain on the set of non-negative integers $\mathbb{Z}_{+}$ that models the process of adding one by successively updating the digits of a number’s expansion in a given numeration system. At each step of the addition algorithm--namely, whenever a carry is performed--a random failure may occur, interrupting the procedure and preventing it from continuing beyond that point.

Killeen and Taylor introduced the first such model in the binary base, acting either on $\Z_{+}$ or on the group $\Z_{2}$ of $2$-adic integers. They showed that, in both cases, the spectrum of the corresponding transition operator is determined by the filled Julia set of the quadratic polynomial map $f\lc z \rc  = \lc z - 1 + p \rc^{2} / p^{2}$. In particular, the spectrum of the operator acting on $\ell^{\oo}\lc \Z_{+} \rc$ coincides exactly with this filled Julia set.

Stochastic adding machines based on more general numeration systems, or with step-dependent error probabilities, have since been studied in \cite{cmv, msv, ms, mv}. In these settings, the spectra of the corresponding transition operators are linked to one-dimensional non-autonomous Julia sets.

The spectral analysis of transition operators provides fundamental insight into their dynamical behavior on separable Banach spaces (see, e.g., \cite{Bayart_Matheron-Dynamics_of_Linear_Operators, Grosse_Peris-Linear_Chaos}). In particular, if a linear operator $S$ is hypercyclic—that is, it admits a dense orbit—or supercyclic—it admits an orbit whose projection on the unit sphere is dense—then each connected component of its spectrum must intersect the unit circle (in the hypercyclic case) or a circle of radius $r \ge 0$ centered at the origin (in the supercyclic case).

\paragraph{} 

A \textit{stochastic adding machine with fallible counter} based on a Cantor system of numeration is a stochastic process, denoted by $\textrm{AMFC}_{\overline{d}, \overline{p}}$, defined for a given sequence of positive integers and a sequence of probabilities satisfying 
\begin{equation}\label{Eq1_Sequence of Integers and Probabilities}
	\overline{d} = \lc d_{r} \rc_{r \ge 1}, \quad d_{r} \ge 2 \text{ for all } r \ge 1,
	\qquad \text{and} \qquad
	\overline{p} = \lc p_{r} \rc_{r \ge 1}, \quad p_{r} \in (0,1] \text{ for all } r \ge 1. 
\end{equation}
Here, $\overline{d}$ encodes the algorithmic structure of the machine, while $\overline{p}$ specifies its stochastic behavior, with $p_{r}$ representing the probability that the $r$-th step of the algorithm is successfully executed.

This model, first analyzed by the authors in \cite{mv}, extends in a natural way the binary stochastic adding machine introduced by Killeen and Taylor.  

In particular, the sequence $\overline{d}$ determines the addition algorithm underlying $\textrm{AMFC}_{\overline{d}, \overline{p}}$ through the \textit{Cantor system of numeration} $\mathrm{CSN}_{\overline{d}}$ (see \cite{mv}), which generalizes the classical $d$-adic numeration systems with $d \ge 2$. The system $\textrm{CSN}_{\overline{d}}$ defines a bijection between $\Z_{+}$ and the set  
\begin{equation}\label{Eq_Space CSN} 
	\mc{Z}_{\overline{d}} :=  \; \lfp \lc a_{r} \rc_{r \ge 1} : \; a_{r} \in \lfp 0, \dots ,d_{r}-1 \rfp \; \text{ for all }\; r \ge 1, \text{ and } \; \sum_{r=1}^{+\oo} a_{r} < +\oo \rfp , 
\end{equation}
so that each $n \in \Z_{+}$ admits a unique \textit{$\textrm{CSN}_{\overline{d}}$-expansion} 
\begin{equation}\label{Eq1_q-Expansion}
	n \; = \; \sum_{r=1}^{+\oo} a_{r}\lc n \rc q_{r-1}, \quad  a_{r}\lc n \rc \in \lfp 0, \dots, d_{r}-1 \rfp \text{ is the }r\text{-th digit of } n, \quad  q_{0} = 1, \; q_{r} = \prod_{s=1}^{r} d_{s} .  
\end{equation}

In $\cite{mv}$, the authors proved that the spectra of the transition operator $S_{\overline{d}, \overline{p}}$--associated with the stochastic adding machine $\textrm{AMFC}_{\overline{d}, \overline{p}}$--acting on any of the classical Banach spaces of sequences $c\lc \Z_{+} \rc$, $c_{0}\lc \Z_{+} \rc$, and $\ell^{q}\lc \Z_{+} \rc$ (for $1 \le q \le +\oo$) coincides with the non-autonomous filled Julia set 
\begin{equation}\label{Eq1_non-autonomous Filled Julia Set}
	\mc{E}_{\overline{d},\overline{p}} := \; \lfp z \in \mathbb{C} : \;  \lfp \tilde{f}_{r}\lc z \rc \rfp_{r \ge 1}  \; \text{is bounded} \rfp  ,     
\end{equation}
where  $\tilde{f}_{r} := f_{r} \circ... \circ f_{1}$ for all $r \ge 1$, and each polynomial $f_{r}:  \mathbb{C} \to \mathbb{C}$ is defined by 
\begin{equation}\label{Eq1_Julia Polynomials}
	f_{r}\lc z \rc := \; \lc { z- \lc 1 - p_{r} \rc \over p_{r} } \rc^{d_{r}} .     
\end{equation}
This result generalizes the first of the two main results of Killeen and Taylor. 
Specifically, their \textit{binary stochastic adding machine} corresponds to $\textrm{AMFC}_{\overline{2}, \overline{p}}$, where $d_{r} = 2$ and $p_{r} = p$ for all $r \ge 1$. With this notation, the associated operator  $S_{\bar{2}, \overline{p}}$ has spectrum equal to the filled Julia set $\mc{E}_{\overline{2}, \overline{p}}$.

\paragraph{}

This paper develops further the investigation initiated in \cite{mv} by providing a full characterization of the spectrum of the transition operator associated with $\textrm{AMFC}_{\overline{d}, \overline{p}}$ acting on the closure of $\mc{Z}_{\overline{d}}$, in direct analogy with the action of Killeen and Taylor’s binary stochastic adding machine on $\Z_{2}$. This closure is the space 
\begin{equation*}
	\Z_{\overline{d}} \; = \; \lfp \lc a_{r} \rc_{r=1}^{+\oo} :\; a_{r} \in \lfp 0 , \dots, d_{r} -1 \rfp  \text{ for all } r \ge 1 \rfp 
\end{equation*} 
endowed with the product topology, where each coordinate space $\lfp 0, \dots, d_{r} -1 \rfp$ carries the discrete topology.

The space $\Z_{\overline{d}}$ is compact and metrizable. Each element $x = \lc a_{r}\lc x \rc \rc_{r \ge 1} \in \Z_{\overline{d}}$ is thus represented by its formal $\textrm{CSN}_{\overline{d}}$-expansion, $ x  \; = \; \sum_{r=1}^{+\oo} a_{r}\lc x \rc q_{r-1} $. 
Under this representation, the set $ \lc \Z_{\overline{d}}, + \rc$ forms an abelian group whose addition law is naturally inherited from $\Z_{+}$. For instance, when $d_{r} = 2$ for all $r\ge 1$, the space $\Z_{\overline{d}} = \lfp 0,1 \rfp^\mathbb{N}$ corresponds to the group of $2$-adic integers $\Z_{2}$.

\paragraph{}

In this paper, a complete description is established for the spectrum of the transition operator $\tilde{S}_{\overline{d}, \overline{p}}: \mc{C}\lc \Z_{\overline{d}} \rc \longrightarrow \mc{C}\lc \Z_{\overline{d}} \rc$ associated with the Markov chain $\textrm{AMFC}_{\overline{d}, \overline{p}}$ on $\Z_{\overline{d}}$. The spectrum is expressed in terms of the non-autonomous filled Julia set $\mc{E}_{\overline{d}, \overline{p}}$ introduced in \eqref{Eq1_non-autonomous Filled Julia Set}. Throughout, $\mc{C}\lc \Z_{\overline{d}} \rc$ denotes the Banach space of continuous complex-valued functions on $\Z_{\overline{d}}$, endowed with the supremum norm.

\begin{theorem}\label{Theor_Conditional Main Result}
	Let $\overline{d}= \lc d_{r} \rc_{r \ge 1}$ and $\overline{p} = \lc p_{r} \rc_{r \ge 1}$ be sequences of positive integers and probabilities, respectively, as defined in equation \eqref{Eq1_Sequence of Integers and Probabilities}. Then the spectrum of the operator $\tilde{S}_{\overline{d}, \overline{p}}: \mc{C}\lc \Z_{\overline{d}} \rc \longrightarrow \mc{C}\lc \Z_{\overline{d} } \rc$ satisfies 
	\begin{equation}\label{Eq__Theor_Conditional Main Result}
		\sigma \lc \tilde{S}_{\overline{d},\overline{p}} \rc \; = \; \sigma_{ap}\lc \tilde{S}_{\overline{d},\overline{p}} \rc \;  = \; \begin{cases}
			\mc{E}_{\overline{d}, \overline{p}}   & , \text{ if } \prod_{r=1}^{+\oo} p_{r} \; = \; 0 , \\
			\partial \mc{E}_{\overline{d}, \overline{p}}   & , \text{ if } \prod_{r=1}^{+\oo} p_{r} \; > \; 0  .
		\end{cases}  \quad \text{and} \quad \partial \mc{E}_{\overline{d}, \overline{p}} \; = \; \overline{ \bigcup_{r=1}^{+\oo} \tilde{f}^{-1}\lfp 1 \rfp}. 
	\end{equation} 
	Moreover, if the sequence $\overline{d}$ is bounded, then the point spectrum satisfies $\sigma_{pt} \lc \tilde{S}_{\overline{d}, \overline{p}} \rc = \bigcup_{r=1}^{\infty} \tilde{f}_{r}^{-1}\lfp 1 \rfp$. 
	
\end{theorem}   

In the special case where $d_{r} = d$ and $p_{r} = p$ for all $r \ge 1$, with $d \ge 2$ and $p \in \lc 0, 1 \rc$, Theorem \ref{Theor_Conditional Main Result} yields $ \sigma \lc \tilde{S}_{\overline{d}, \overline{p}} \rc  =  \mc{E}_{\overline{d}, \overline{p}} $ and $ \overline{\sigma_{pt} \lc \tilde{S}_{\overline{d}, \overline{p}} \rc}  =  \partial \mc{E}_{\overline{d}, \overline{p}} $. In particular, when $d_{r} =2$ for all $r \ge 1$, this recovers the inclusion $\sigma\lc \tilde{S}_{\bar{2}, \overline{p} } \rc  \sub \mc{E}_{\overline{2}, \overline{p}}$ originally established by Killeen and Taylor \cite{Killeen_Taylor-A_Stochastic Adding_Machine_And_Complex_Dynamics}. 

As for the case $\prod_{r=1}^{+\oo} p_{r} > 0$, Theorem \ref{Theor_Conditional Main Result} reveals a phenomenon that, to the best of our knowledge, has not been observed before: the spectrum of a stochastic adding machine coincides precisely with the boundary of a non-autonomous filled Julia set.
The proof of this fact is subtle and technically involved, relying in an essential way on the complex dynamical structure of the underlying non-autonomous iteration.

\paragraph{}

As a consequence of Theorem \ref{Theor_Conditional Main Result}, one obtains the following result concerning the dynamical properties of the operator $\tilde{S}_{\overline{d}, \overline{p}}$.

\begin{corollary}\label{Cor_Dynamical Properties of Transition Operators}
	Let $\overline{d}= \lc d_{r} \rc_{r \ge 1}$ and $\overline{p} = \lc p_{r} \rc_{r \ge 1}$ be sequences of positive integers and probabilities, respectively, as defined in equation \eqref{Eq1_Sequence of Integers and Probabilities}. Then the following statements hold: 
	\begin{enumerate}
		\item If $\sigma \lc \tilde{S}_{\overline{d}, \overline{p}} \rc$ is homeomorphic to a Cantor set, then the operator $\tilde{S}_{\overline{d}, \overline{p}}: \mc{C}\lc \Z_{\overline{d}} \rc \longrightarrow \mc{C}\lc \Z_{\overline{d}} \rc$  is not supercyclic.  
		\item If $\prod_{r=1}^{+\oo} p_{r} > 1/2$, then $\kappa \tilde{S}_{\overline{d}, \overline{p}}$ is not hypercyclic for all $ \kappa \ge (2\prod_{r = 1}^{+\oo} p_{r} -1)^{-1}$. 
	\end{enumerate}
	
\end{corollary}

Dynamical properties of transition operators, such as supercyclicity and (frequent) hypercyclicity, have also been studied in \cite{Messaoudi_Valle-Dynamical_Properties_of_Random_Walks}. In that work, the authors investigate stochastic operators arising from random walks on $\Z_{+}$ acting on the spaces $\ell^{q}\lc \Z_{+} \rc$ (for $1 \le q \le +\infty$) and $c_{0}\lc \Z_{+} \rc$.

Regarding Point $(1$) of Corollary \ref{Cor_Dynamical Properties of Transition Operators}, sufficient conditions on the sequences $\overline{d}$ and $\overline{p}$ ensuring that the non-autonomous Julia set $\mc{E}_{\overline{d}, \overline{p}}$--and hence, by Theorem \ref{Theor_Conditional Main Result}, the spectrum $\sigma\lc \tilde{S}_{\overline{d}, \overline{p}} \rc$--is a Cantor set are given in \cite[Section 4]{msv}. For instance, if $p_{r} = p < 1/2$  and $d_{r} = 2d$ for all $r \ge 1$, then $\mc{E}_{\overline{d}, \overline{p}}$ is a Cantor set \cite[Proposition 4.4]{msv}.

\paragraph{}
It is worth mentioning that Cantor numeration systems are closely related to the Vershik maps on certain Bratteli diagrams  \cite{Bezuglyi_Karpel-Bratelli_Diagrams_structure_measures_dynamics}. Extending the current analysis to stochastic adding machines associated with Vershik maps would be of particular interest, as such systems play a significant role in operator algebras and dynamical systems (see, e.g., \cite{Giordano_Putnam_Ckau-Topological_Orbit_Equivalence_and_C*}, \cite{Herman_Putnam_Skau-Ordered_Bratteli_Diagrams_Dimension_Groups_and_Topological_Dynamics}, \cite{Medynets-Cantor_Aperiodic_Systems_and_Bratelli_Diagramms}, \cite{Vershik-Uniform_Algebraik_Approximation_of_Shift_and_Multiplication_Operators}).

\paragraph{}

The paper is organized as follows. Section \ref{Sec_Defiition of the Operator} discusses preliminary notions and provides detailed definitions of the stochastic adding machine $\textrm{AMFC}_{\overline{d}, \overline{p}}$ and the transition operator $\tilde{S}_{\overline{d}, \overline{p}}: \mc{C}\lc \Z_{\overline{d}} \rc \longrightarrow \mc{C}\lc \Z_{\overline{d}} \rc$. 
The proofs of the main results are presented in Section \ref{Sec_Proof of the Main Results}, while Section \ref{Sec_Appendix} (Appendix) contains the proofs of two auxiliary lemmas.

\section{Definition of the Stochastic Adding Machine}\label{Sec_Defiition of the Operator}

The stochastic process $\textrm{AMFC}_{\overline{d}, \overline{p}}$ on $\Z_{+}$ models the operation of adding one to a non-negative integer $n \in \Z_{+}$ expressed in its $\textrm{CSN}_{\overline{d}}$-expansion. In the deterministic setting, the addition algorithm with respect to the Cantor system of numeration $\textrm{CSN}_{\overline{d}}$ assigns to $n + 1$ the digits 
\begin{equation}\label{Eq1_Counter}
	a_{r}\lc n+1 \rc \; = \;  \begin{cases}
		0 &,  \text{ if } \ r< s_{n},  \\
		a_{r}\lc n \rc + 1 &,  \text{ if } \ r = s_{n},  \\
		a_{r}\lc n \rc &,  \text{ if } \ r >\ s_{n},   
	\end{cases} \quad \text{where }  s_{n} \;= \; \min\lfp r\ge 1: \; a_{r}\lc n \rc \neq d_{r}-1 \rfp   ,
\end{equation}
where the integer $s_{n}$ is called the \textit{counter} of $n \in \Z_{+}$.

In the stochastic version of the addition algorithm, at each step $r$, the information about the counter may be lost independently with probability $1 - p_{r}$, causing the addition to terminate prematurely. As a result, the output of the machine is a random variable. Consequently, the process  $\textrm{AMFC}_{\overline{d}, \overline{p}}$ determines the \textit{transition probability} $\textrm{p}\lc n, m \rc = \textrm{p}_{\overline{d}, \overline{p}}\lc n, m \rc$ from state $n$ to $m$  given by 
\begin{equation}\label{Eq1_Transition Probabilities Integers}
	\textrm{p}\lc n, m \rc := \;
	\begin{cases}
		\lc 1-p_{s +1} \rc \prod_{r=1}^{s} p_{r} &,  \text{ if }  m = \sum_{r = s + 1}^{+\oo} a_{r}\lc n \rc q_{r-1}   \text{ with } 1\le s \le s_{n} -1 , \\ 
		1-p_1 &,  \text{ if } \ m=n \, , \\
		\prod_{r=1}^{s_n} p_{r} &,  \text{ if } \ m=n+1 \, , \\
		0 &,  \textrm{ otherwise} .  
	\end{cases} 
\end{equation} 
Accordingly, given $n \in \Z_{+}$, the machine performs without error and outputs $n+1$ with probability $\prod_{r=1}^{s_{n}} p_{r}$, whereas, with the complementary probability $1 - \prod_{r=1}^{s_{n}} p_{r}$, the addition process halts early, yielding an output obtained from $n$ by truncating some of its first digits.

\paragraph{}

The stochastic adding machine $\textrm{AMFC}_{\overline{d}, \overline{p}}$ acting on $\Z_{\overline{d}}$ is a discrete-time Markov process that inherits probabilistic characteristics and operational structure from its counterpart on $\Z_{+}$. For instance, it is topologically irreducible and exhibits either null recurrence, when $\prod_{r=1}^{+\infty} p_r = 0$, or transience, when $\prod_{r=1}^{+\infty} p_r > 0$ \cite[Proposition 2.1]{mv}. From an operational aspect, the essential distinction lies in the sequence $\lc d_{r} - 1 \rc_{r \ge 1} \in \Z_{\overline{d}}$, which has the identically zero sequence as its successor since $1 + \sum_{r=1}^{+\oo} \lc d_{r} - 1 \rc q_{r-1} = 0$.

\paragraph{}

The process $\textrm{AMFC}_{\overline{d}, \overline{p}}$, when acting respectively on $\Z_{+}$ and on $\Z_{\overline{d}}$, determines two \textit{transition operators}:
$$ S_{\overline{d}, \overline{p}} : \ell^{\oo}\lc \Z_{+} \rc   \; \longrightarrow  \;  \ell^{\oo}\lc \Z_{+} \rc \quad \text{and} \quad \tilde{S}_{\overline{d}, \overline{p}}: \mc{C}\lc \Z_{\overline{d}} \rc \; \longrightarrow \; \mc{C}\lc \Z_{\overline{d}} \rc. $$
Here, $\ell^{\oo}\lc \Z_{+} \rc $ denotes the Banach space of bounded complex sequences, and $\mc{C}\lc \Z_{\overline{d}} \rc$ is the space of continuous complex-valued functions on $\Z_{\overline{d}}$, both endowed with the supremum norm.

More precisely, the operator $S_{\overline{d}, \overline{p}}$ is represented by the \textit{transition matrix} $S_{\overline{d},\overline{p}} = \lb \textrm{p}\lc n, m \rc \rb_{n,m \ge 0}$, where the entries $\textrm{p}\lc n, m \rc$ are the transition probabilities defined in Equation \eqref{Eq1_Transition Probabilities Integers}. 
The matrix $S_{\overline{d}, \overline{p}}$ is stochastic, since the entries in each row sum to $1$. 
Moreover, the column sums are also equal to one, except for the first column, whose total sum equals $1- \prod_{r=1}^{+\oo} p_{r} $. 
Hence, $S_{\overline{d}, \overline{p}}$ is \textit{doubly stochastic} if and only if $\prod_{r=1}^{+\oo} p_{r} = 0$. 

$$
\tiny{
	\left[
	\begin{array}{ccccccccccc}
		\!\!1-p_1        \!\!&\!\!p_1  \!\!&\!\!0    \!\!&\!\!0         \!\!&\!\!0    \!\!&\!\!0    \!\!&\!\!0         \!\!&\!\!0    \!\!&\!\!0    \!\!&\!\!0  \!\!&\!\! \cdots \!\! \\
		\!\!0            \!\!&\!\!1-p_1\!\!&\!\!p_1  \!\!&\!\!0         \!\!&\!\!0    \!\!&\!\!0    \!\!&\!\!0         \!\!&\!\!0    \!\!&\!\!0    \!\!&\!\!0  \!\!&\!\! \cdots \!\! \\
		\!\!p_1(1-p_2)   \!\!&\!\!0    \!\!&\!\!1-p_1\!\!&\!\!p_1p_2    \!\!&\!\!0    \!\!&\!\!0    \!\!&\!\!0         \!\!&\!\!0    \!\!&\!\!0    \!\!&\!\!0  \!\!&\!\!\cdots \!\! \\
		\!\!0            \!\!&\!\!0    \!\!&\!\!0    \!\!&\!\!1-p_1     \!\!&\!\!p_1  \!\!&\!\!0    \!\!&\!\!0         \!\!&\!\!0    \!\!&\!\!0    \!\!&\!\!0  \!\!&\!\! \cdots \!\! \\
		\!\!0            \!\!&\!\!0    \!\!&\!\!0    \!\!&\!\!0         \!\!&\!\!1-p_1\!\!&\!\!p_1  \!\!&\!\!0         \!\!&\!\!0    \!\!&\!\!0    \!\!&\!\!0  \!\!&\!\! \cdots \!\! \\
		\!\!0            \!\!&\!\!0    \!\!&\!\!0    \!\!&\!\!p_1(1-p_2)\!\!&\!\!0    \!\!&\!\!1-p_1\!\!&\!\!p_1p_2    \!\!&\!\!0    \!\!&\!\!0    \!\!&\!\!0  \!\!&\!\! \cdots \!\! \\
		\!\!0            \!\!&\!\!0    \!\!&\!\!0    \!\!&\!\!0         \!\!&\!\!0    \!\!&\!\!0    \!\!&\!\!1-p_1     \!\!&\!\!p_1  \!\!&\!\!0    \!\!&\!\!0  \!\!&\!\! \cdots \!\! \\
		\!\!0            \!\!&\!\!0    \!\!&\!\!0    \!\!&\!\!0         \!\!&\!\!0    \!\!&\!\!0    \!\!&\!\!0         \!\!&\!\!1-p_1\!\!&\!\!p_1  \!\!&\!\!0  \!\!&\!\! \cdots \!\! \\
		\!\!p_1p_2(1-p_3)\!\!&\!\!0    \!\!&\!\!0    \!\!&\!\!0         \!\!&\!\!0    \!\!&\!\!0    \!\!&\!\!p_1(1-p_2)\!\!&\!\!0    \!\!&\!\!1-p_1\!\!&\!\!p_1p_{2}p_{3}\!\!&\!\! \cdots \!\! \\
		\!\!\vdots \!\!&\!\!\vdots \!\!&\!\!\vdots \!\!&\!\!\vdots \!\!&\!\!\vdots \!\!&\!\!\vdots \!\!&\!\!\vdots \!\!&\!\!\vdots \!\!&\!\!\vdots \!\!&\!\!\vdots \!\!&\!\! \ddots \!\!
	\end{array}
	\right]} 
$$
\begin{center}
	Transition matrix $S_{\overline{d}, \overline{p}}$ for the case $d_{r} =3$ (base $3$). 
\end{center}

\paragraph{}
The \textit{transition operator} $\tilde{S}_{\overline{d}, \overline{p}}: \mc{C}\lc \Z_{\overline{d}} \rc \longrightarrow \mc{C} \lc \Z_{\overline{d}} \rc$ associated with $\textrm{AMFC}_{\overline{d}, \overline{p}}$ acting on $\Z_{\overline{d}}$ is defined by 
\begin{equation*}
	\lc \tilde{S}_{\overline{d},\overline{p}} g \rc \lc x \rc :=  \; \sum_{y \in \Z_{\overline{d}} } \textrm{p}\lc x, y \rc \cdot g\lc y \rc , \quad \text{for all } g \in \mc{C}\lc \Z_{\overline{d}} \rc \text{ and } x \in \Z_{\overline{d}} , 
\end{equation*}
where $\textrm{p}\lc x,y \rc = \textrm{p}_{\overline{d}, \overline{p}}\lc x, y \rc$ denotes the \textit{transition probability} from $x$ to $y$. These probabilities are determined analogously to those in equation \eqref{Eq1_Transition Probabilities Integers}, with $x$ and $y$  replacing $n$ and $m$, respectively. In this case, for each $x \in \Z_{\bar{d}}$, the \textit{counter} $s_{x} $ (a natural number or $+\oo$) is defined as   
\begin{equation*}\label{Eq1_Counter Continuous State Space}
	s_{x} \;= \; s_{x, \overline{d}} := \; \min\lfp 1\le r \le +\oo: \; a_{r}\lc x \rc \neq d_{r}-1 \rfp .
\end{equation*}
Substituting the probabilistic relation from \eqref{Eq1_Transition Probabilities Integers} yields the following expression for the operator: 
\begin{equation}\label{Eq3_Transition Operator}
	\lc \tilde{S}_{\overline{d},\overline{p}} g \rc \lc x \rc   \; {=} \; \lc 1 - p_{1} \rc g\lc x \rc + \lc \prod_{r=1}^{s_{x}} p_{r} \rc g\lc x + 1 \rc + \sum_{s = 1}^{s_{x} -1 } \lc \prod_{r=1}^{s} p_{r} \rc \lc 1 -p_{s + 1} \rc g\lc  T_{s}\lc x \rc \rc ,
\end{equation}
where $T_{s}$ is defined by  
\begin{equation}\label{Eq3_Trancation Integers}
	T_{s}\lc x \rc := \; x - \sum_{r=1}^{s} \lc d_{r} -1 \rc q_{r-1} \; = \; \sum_{r= s+1}^{+\oo} a_{r}\lc x\rc q_{r-1}    \quad \text{for }1 \le s \le s_{x}.     
\end{equation}

Relation \eqref{Eq3_Transition Operator} is a well-defined linear operator whose norm satisfies $ || \tilde{S}_{\overline{d}, \overline{p} } || \le 1$. The bound follows from the identity  
$$ \lc 1 - p_{1} \rc + \prod_{r = 1}^{+\oo} p_{r} + \sum_{s = 1}^{+\oo} \lc \lc 1 - p_{s + 1} \rc \prod_{r=1}^{s} p_{r} \rc \; = \; 1 .$$ 
When $s_{x} =1$,  the operator reduces to the simple form  $ \lc \tilde{S}_{\overline{d}, \overline{p}} \rc g\lc x \rc \; = \; \lc 1 - p_{1} \rc g\lc x \rc  + p_{1} g\lc x + 1 \rc $, whereas, when $x = -1$ (i.e., $ x= \lc d_{r} -1 \rc_{r \ge 1}$ and $s_{x} = +\oo$), its actions is given by 
$$ \lc \tilde{S}_{\overline{d}, \overline{p}} \rc g\lc -1 \rc \; = \; \lc 1 - p_{1} \rc g\lc -1 \rc + \lc \prod_{r=1}^{\oo } p_{r} \rc g\lc 0 \rc + \sum_{s = 1}^{\oo } \lc \prod_{r=1}^{s} p_{r} \rc \lc 1 -p_{s + 1} \rc g\lc  T_{s}\lc -1 \rc \rc .$$

Throughout the paper, it will often be convenient to employ the following equivalent analytical definition of the operator $\tilde{S}_{\overline{d}, \overline{p}}$:
\begin{equation}\label{Eq3_Projection of the Stochastic Operator}
	\lc \tilde{S}_{\overline{d}, \overline{p}}g \rc \lc n \rc \; = \;  \lc S_{\overline{d}, \overline{p}} \circ P g \rc\lc n \rc \quad \text{for all }\; n \in \Z_{+} ,
\end{equation}
where the operator $ P : \mc{C}\lc \Z_{\overline{d}} \rc \longrightarrow \ell^{\oo}\lc \Z_{+} \rc$ is defined by $Pg := \lc g\lc n \rc \rc_{n \in \Z_{+}}$. Since $\Z_{+}$ is dense in $\Z_{\overline{d}}$, every continuous function $g \in \mc{C}\lc \Z_{\overline{d}} \rc$ is uniquely determined by its values on $\Z_{+}$. Therefore, the operator $P$ provides a natural embedding of $\mc{C}\lc \Z_{\overline{d}} \rc$ into $\ell^{\oo}\lc \Z_{+} \rc$. 
Consequently, relation \eqref{Eq3_Transition Operator} is recovered from \eqref{Eq3_Projection of the Stochastic Operator}.

\paragraph{}

This formulation establishes the analytical groundwork for the spectral study of $\tilde{S}_{\overline{d}, \overline{p}}$, which will be pursued in the next section. Recall that the spectrum $\sigma\lc T \rc$ of a bounded operator $T$ over a complex Banach space $T$ is the set of complex numbers $\lambda$ for which the operator $ T - \lambda I $ is not invertible, where $I$ denotes the identity operator on $X$. In particular, the \textit{point spectrum} (or set of eigenvalues) and the \textit{approximate spectrum} are given, respectively, by 
$$ \sigma_{pt}\lc T \rc  := \lfp \lambda \in \mathbb{C} :
T - \lambda I  \text{ is not injective} \rfp$$ and $$\sigma_{ap}\lc T \rc \; = \; \lfp \lambda \in \C : \; \underset{\ldav x \rdav = 1}{\inf} \ldav \lc T - \lambda I \rc x \rdav \; = \; 0 \rfp .$$ 
Recall also that an operator $T$ is said to be \textit{hypercyclic} if there exists $x \in X$ such that the closure of its orbit $\lfp T^{n}\lc x \rc : \; n \in \N \rfp$ is dense in $X$. The operator $T$ is said to be \textit{supercyclic} if there exists $x \in X$ such that its projective orbit $ \lfp z T^{n}\lc x \rc : \; n \in \N, \; z \in \C \rfp $ is dense in $X$.

\section{Proof of the main result}\label{Sec_Proof of the Main Results}

This section is structured as follows. Section \ref{Section_Proof Unconditional Results} establishes several spectral properties of the operator $\tilde{S}_{\overline{d}, \overline{p}}$ that hold independently of the specific choice of the sequences $\overline{p}$. 
These general properties are summarized in Proposition \ref{Prop_Unconditional Main Result} and constitute a fundamental ingredient in the proof of Theorem \ref{Theor_Conditional Main Result}. Section \ref{Section_Proof Conditional Results} then completes the proof of Theorem \ref{Theor_Conditional Main Result} under the additional assumptions imposed on the sequences $\overline{p}$. Finally, Section \ref{Sec_Cor_Dynamical Properties of Transition Operators} contains the proof of Corollary \ref{Cor_Dynamical Properties of Transition Operators}.

\begin{proposition}\label{Prop_Unconditional Main Result}
	Let $\overline{d}= \lc d_{r} \rc_{r \ge 1}$ and $\overline{p} = \lc p_{r} \rc_{r \ge 1}$ be sequences of positive integers and probabilities, respectively, as defined in equation \eqref{Eq1_Sequence of Integers and Probabilities}. Then the following assertion holds:  
	\begin{equation}\label{Eq_Theor_Unconditional Result}
		\partial \mc{E}_{\overline{d}, \overline{p}} \; \sub \; \sigma \lc \tilde{S}_{\overline{d}, \overline{p}} \rc \; \sub \;  \mc{E}_{\overline{d},\overline{p}}.
	\end{equation}
	In particular, the union $\bigcup_{r=1}^{\infty} \tilde{f}_{r}^{-1}\lfp 1 \rfp$ is contained in the point spectrum $\sigma_{pt} \lc \tilde{S}_{\overline{d}, \overline{p}} \rc $, and its closure coincides with the boundary of $ \mc{E}_{\overline{d}, \overline{p}}$, that is, 
	\begin{equation}\label{Eq_Theor_Unconditional Point Spectrum Inclusion} 
		\overline{} \bigcup_{r=1}^{\infty} \tilde{f}_{r}^{-1}\lc \lfp 1 \rfp \rc \; \sub \; \sigma_{pt} \lc \tilde{S}_{\overline{d}, \overline{p}} \rc  \quad \text{and} \quad \overline{\bigcup_{r=1}^{+\oo} \tilde{f}_{r}^{-1}\lc \lfp 1 \rfp \rc} \; = \; \partial \mc{E}_{\overline{d}, \overline{p}}  .
	\end{equation}
	Moreover, if the sequence $\overline{d}$ is bounded, then 
	\begin{equation}\label{Eq_Point Spectrum Bounded Cantor System}
		\sigma_{pt} \lc \tilde{S}_{\overline{d}, \overline{p}} \rc \;  = \; \bigcup_{r=1}^{\infty} \tilde{f}_{r}^{-1}\lfp 1 \rfp .
	\end{equation}

\end{proposition}

Throughout the proof, the non-autonomous filled Julia set $\mc{E}_{\overline{d}, \overline{p}}$ as well as the complex polynomials $\tilde{f}_{r}$ and $f_{r}$ for $r \ge 1$ are defined as in equations \eqref{Eq1_non-autonomous Filled Julia Set} and \eqref{Eq1_Julia Polynomials}, respectively.
For convenience, $\tilde{f}_{0} = f_{0}$ is set to be the identity map on $\C$.

\subsection{Proof of Proposition \ref{Prop_Unconditional Main Result}.}\label{Section_Proof Unconditional Results} 

In this section, we establish the relations \eqref{Eq_Theor_Unconditional Point Spectrum Inclusion}, \eqref{Eq_Point Spectrum Bounded Cantor System}, and \eqref{Eq_Theor_Unconditional Result}, in that order.

\paragraph{}
Let $\lambda \in \C$ be a complex number. Define the complex-valued sequence $ \bm{v}_{\lambda}  : \Z_{+} \to \C$ by 
\begin{equation}\label{ps1}
	v_\lambda (n) := \; \prod_{r=1}^{\infty} \lc \iota_{\lambda}\lc r \rc \rc^{a_{r}\lc n \rc} \quad \text{for } \; n \in \Z_{+}, 
\end{equation}
where $a_{r}\lc n \rc$ denotes the $r$-th digit of $n$ in its $\textrm{CSN}_{\overline{d}}$-expansion. By convention, one sets $\lc \iota_{\lambda}\lc r \rc \rc^{a_{r}(n)} = 1$ whenever $\iota_{\lambda}\lc r \rc = 0$ and $a_{r}\lc n \rc = 0$. 
The mapping $\iota_{\lambda}\lc r \rc$ is defined inductively by 
\begin{equation}\label{ps1} 
	\iota_{\lambda}\lc 1 \rc \; = \; {\lambda \over p_{1}} - {1 - p_{1} \over p_{1}} \quad \text{and} \quad  \iota_{\lambda}\lc r \rc \; = \; {1\over p_{r}} \iota_{\lambda}\lc r-1 \rc^{d_{r-1}}  - {1 - p_{r} \over p_{r}}, \; \text{for } r \ge 2   . 
\end{equation}
Observe that, equivalently, one may write  
\begin{equation}\label{ps2}
	\iota_{\lambda}\lc r \rc := \; \lc  h_r \comp \tilde{f}_{r-1}  \rc  \lc \lambda \rc , \quad \text{where} \quad \; h_{r} \lc z \rc := \; \frac{z}{p_r} - \frac{1-p_r}{p_r} \quad \text{for } r \ge 1 .
\end{equation}

\begin{example}
	Consider the case where $d_{r} = 2$ and $p_{r} = p$ for all $r \ge 1$, with $0 \;< p \;< 1$. 
	In this setting, the sequence $v_\lambda \lc n \rc$ simplifies to $\prod_{r=1}^{+\oo} \lc \iota_{\lambda}\lc r \rc \rc^{a_{r}\lc n \rc}$, where $a_{r}\lc n \rc$ are the binary digits of $n$, that is, $n = \sum_{r=1}^{\oo} a_{r}\lc n \rc 2^{r-1}$. In particular,  
	$$ \iota_{\lambda}\lc 1 \rc  \; = \; {\lambda - 1 + p \over p}, \quad \text{and for any } r\ge 1, \quad \iota_{\lambda}\lc r+ 1 \rc \; = \; {1\over p}\iota_{\lambda}\lc r \rc^{2} - {1 - p \over p} .$$ 
	
\end{example}

\paragraph{}
To establish the left-hand side inclusion in \eqref{Eq_Theor_Unconditional Point Spectrum Inclusion}, the following lemma is employed. It asserts that the non-autonomous filled Julia set $\mc{E}_{\overline{d}, \overline{p}}$ coincides with the point spectrum of the operator $S_{\overline{d}, \overline{p}}$, and it provides equivalent characterizations of $\mc{E}_{\overline{d}, \overline{p}}$ in terms of the sequences $\lfp \iota_{\lambda}\lc r \rc \rfp_{r \ge 1}$ and $\lfp v_{\lambda}\lc n \rc \rfp_{n \in \Z_{+}}$.

The proof is omitted, as a version of this result was previously established in \cite[Proposition 3.4 \& Lemma 3.6]{mv}. The statement presented here, however, offers a more detailed formulation; the additional assertions can be verified straightforwardly by the reader.

\begin{lemma}\label{Lem_Structure of non-autonomous Julia Set 1}
	Let $\overline{d}= \lc d_{r} \rc_{r \ge 1}$ and $\overline{p} = \lc p_{r} \rc_{r \ge 1}$ be as in relation \eqref{Eq1_Sequence of Integers and Probabilities}.   
	Then the non-autonomous filled Julia set $\mc{E}_{\overline{d}, \overline{p}}$ coincides with the point spectrum $  \sigma_{pt}\lc S_{\overline{d}, \overline{p}} \rc$. 
	
	More precisely,  a vector $ \bm{w} \in \ell^{\oo}\lc \Z_{+} \rc$ is an eigenvector of $S_{\overline{d}, \overline{p}}$ associated with an eigenvalue $\lambda$ if and only if  $\lambda \in \mc{E}_{\overline{d}, \overline{p}}$, and there exist a non-zero constant $c \in \C\backslash \lfp 0 \rfp$ such that $\bm{w} = c \cdot \bm{v}_{\lambda}$, where $\bm{v}_{\lambda}$ is defined by \eqref{ps1}.   In particular, the following equivalent characterizations of $\mc{E}_{\overline{d}, \overline{p}}$ hold:
	\begin{align*}
	\mc{E}_{\overline{d},\overline{p}}\;  & = \;	\lfp \lambda \in \mathbb{C} :  \lav \tilde{f}_{r-1}\lc \lambda \rc \rav \leq 1,\; \text{for every } r \geq 1 \rfp \\
	& = \;  \lfp \lambda \in \mathbb{C} : \;  \lav \iota_{\lambda}\lc r \rc \rav \leq 1,\; \text{for every } r \ge 1 \rfp \\
	& = \;  \lfp \lambda \in \mathbb{C} : \;  \lc v_{\lambda} \lc n \rc \rc_{n \in \Z_{+} } \; \text{ is bounded} \rfp ,
	\end{align*}
	where the quantities $\iota_{\lambda}\lc r \rc$ for $r \ge 1$ are given by \eqref{ps2}. 
	Moreover, for any $\lambda \in \C$, the following equivalence holds: 
	\begin{equation}\label{Eq_Divergent Sequences}
		\lim_{r \to +\oo} \lav \tilde{f}_{r}\lc \lambda \rc \rav \; = \; +\oo \quad \text{if and only if} \quad \lav \tilde{f}_{r_{0}}\lc \lambda \rc \rav  \;> 1 \; \text{for some } r_{0} \ge 1 .
	\end{equation}

\end{lemma}

\paragraph{}

\begin{proof}[Proof of the left-hand side inclusion in \eqref{Eq_Theor_Unconditional Point Spectrum Inclusion}]
	We show that $\bigcup_{r=1}^{\infty} \tilde{f}_{r}^{-1}\lfp 1 \rfp \sub \sigma_{pt} \lc \tilde{S}_{\overline{d}, \overline{p}} \rc $. To this end, fix $\lambda \in \bigcup_{r=1}^{+\oo} \tilde{f}_{r}^{-1}\lfp 1 \rfp$ be such that $\tilde{f}_{r_{0}}\lc \lambda \rc=1$. Consequently, $ \tilde{f}_{r}\lc \lambda \rc  =  1$ and $ \iota_{\lambda}\lc r \rc  =  1$  for all $r \ge r_{0}$. It suffices to show that the function $g_{\lambda}: \Z_{\overline{d}} \to \C$ defined by 
	$$ g_{\lambda}\lc x \rc := \; \prod_{r=1}^{\infty} \lc \iota_{\lambda}\lc r \rc \rc^{a_{r}(x)} \; = \;  \prod_{r=1}^{ r_{0}} \lc \iota_{\lambda}\lc r \rc \rc ^{a_{r}\lc x \rc}, \quad \text{for } x= \lc a_{r}\lc x \rc \rc_{r=1}^{+\oo} \in \Z_{\overline{d}} , $$
	is an eigenfunction of $\tilde{S}_{\overline{d}, \overline{p}}$. 
	It is straightforward to verify that $g_{\lambda} $ is continuous and coincides with the sequence $ \bm{v}_{\lambda}: \Z_{+} \to \C$ whenever  $x \in \Z_{+}$.
	Hence, by Lemma \ref{Lem_Structure of non-autonomous Julia Set 1} and relation \eqref{Eq3_Projection of the Stochastic Operator}, we obtain $ \lc \tilde{S}_{\overline{d}, \overline{p}} - \lambda \tilde{I} \rc g_{\lambda} = 0$, where $\tilde{I}: \mc{C}\lc \Z_{\overline{d}} \rc \longrightarrow \mc{C}\lc \Z_{\overline{d}} \rc$ denotes the identity map. This completes the proof.
	
\end{proof}

\paragraph{}

The proof of right-hand side equality in \eqref{Eq_Theor_Unconditional Point Spectrum Inclusion} relies on the following lemma concerning normal families of analytic functions.
Let $\mc{G} = \lfp g_{r} : r \ge 1 \rfp$ denote a family of analytic mappings $g_{r} : W \to \C$ defined on an open set $W \sub \C$.
The family $\mc{G}$ is said to be \textit{normal} at a point $z \in W$ if there exists an open neighborhood $U \sub W$ of $z$ such that every sequence in $\mc{G}$ admits a subsequence converging uniformly on $U$ either to a bounded analytic function or to infinity (see, for instance, \cite[Chapter 9, p. 340]{Markushevich-Theory_of_Functions_of_Complex_Variable_III}).
A sequence $\lc g_{r} : U \to \C \rc_{r \ge 1}$ is said to converge uniformly to infinity if, for every compact subset $K \sub U$ and every $M > 0$, there exists $R_{K,M} \ge 1$ such that $\lav g_{r}(z) \rav \ge M$ for all $z \in K$ and all $r \ge R_{K,M}$.

In this context, the lemma below plays a central role. 
It characterizes the points of non-normality of the family $\lfp \tilde{f}_{r} : r \ge 1 \rfp$ in terms of the boundary of the non-autonomous filled Julia set $\mc{E}_{\overline{d},\overline{p}}$ and provides a geometric property of the preimages $\tilde{f}_{r}^{-1}\lc 1 \rc$. The proof of the lemma is postponed to the appendix.

\begin{lemma}\label{Lem5_Boundary and Non-Normal Points}
	Let $\overline{d}= \lc d_{r} \rc_{r \ge 1}$ be a sequence of positive integers and $\overline{p} = \lc p_{r} \rc_{r \ge 1}$ be a sequence of probabilities. Then the family $\lfp \tilde{f}_{r}: \; r \ge 1 \rfp$ is normal at $z \in \C$ if and only if $z \in  \C \backslash \partial \mc{E}_{\overline{d}, \overline{p}}$.  
	Moreover, if $U \sub \mathbb{C}$ is an open set such that $ U \cap\partial \mc{E}_{\overline{d},\overline{p}} \neq \emptyset$, then there exists $r \ge 1$ such that $\tilde{f}_{r}^{-1}\lfp 1 \rfp \cap U \neq \emptyset $. 
\end{lemma}

\begin{proof}[Proof of the right-hand side equality in \eqref{Eq_Theor_Unconditional Point Spectrum Inclusion}] 
	We first establish the inclusion $\overline{\bigcup_{r=1}^{+\oo} \tilde{f}_{r}^{-1}\lfp 1 \rfp} \sub \partial \mc{E}_{\overline{d}, \overline{p}}$. 
	Since the boundary $\partial \mc{E}_{\overline{d}, \overline{p}}$ is closed, 
	it suffices to prove that $ \bigcup_{r=1}^{\oo} \tilde{f}_{r}^{-1} \lfp 1 \rfp  \sub \partial \mc{E}_{\overline{d}, \overline{p}}$. 
	Let $z \in \bigcup_{r=1}^{\oo} \tilde{f}_{r}^{-1} \lfp 1 \rfp $ be such that $ \tilde{f}_{r_{0}}\lc z \rc = 1$. Since $\tilde{f}_{r_{0}}$ is holomorphic and nonconstant, it is an open mapping. Therefore, for every $\e \;> 0$, there exists $ z_{\e} \in \C$ with $\lav z - z_{\e} \rav \;< \e$ such that $ \lav \tilde{f}_{r_{0}} \lc z_{\e} \rc \rav \;> 1$.  
	By Lemma \ref{Lem_Structure of non-autonomous Julia Set 1}, it follows that $z_{\e}  \not\in \mc{E}_{\overline{d}, \overline{p}}$, while $z \in \mc{E}_{\overline{d}, \overline{p}}$ because $ \tilde{f}_{r}\lc z \rc = 1$ for all $r \ge r_{0}$. Hence $z \in \partial \mc{E}_{\overline{d}, \overline{p}}$, and thus $\bigcup_{r=1}^{+\oo} \tilde{f}_{r}^{-1}\lc \lfp 1 \rfp \rc  \sub \partial \mc{E}_{\overline{d}, \overline{p}}$.

	\medskip
	To prove the reverse inclusion $ \partial \mc{E}_{\overline{d}, \overline{p}} \sub \overline{\bigcup_{r=1}^{\oo} \tilde{f}_{r}^{-1} \lfp 1 \rfp }$, fix a point $z \in \partial \mc{E}_{\overline{d},\overline{p}}$ and let $U_{z} \sub \mathbb{C}$ be an open neighborhood of $z$. 
	By Lemma \ref{Lem5_Boundary and Non-Normal Points}, the intersection $  U_{z} \cap \lc \bigcup_{r=1}^{+\oo} \tilde{f}_{r}^{-1}\lfp 1 \rfp \rc $ is not empty. 
	Since the neighborhood $U_{z}$ is arbitrary, it follows that $z \in \overline{ \bigcup_{r=1}^{+\oo} \tilde{f}_{r}^{-1}\lfp 1 \rfp }$, which completes the proof.   
\end{proof}

\paragraph{}

Having established \eqref{Eq_Theor_Unconditional Point Spectrum Inclusion}, we now assume that the sequence $\overline{d}$ is bounded. Under this assumption, the left-hand inclusion in \eqref{Eq_Theor_Unconditional Point Spectrum Inclusion} becomes the equality stated in \eqref{Eq_Point Spectrum Bounded Cantor System}.

\begin{proof}[Proof of \eqref{Eq_Point Spectrum Bounded Cantor System}] 
	Assume that the sequence $\overline{d}$ is bounded. It suffices to prove that $\sigma_{pt} \lc \tilde{S}_{\overline{d}, \overline{p}} \rc \sub \bigcup_{r=1}^{\infty} \tilde{f}_{r}^{-1}\lfp 1 \rfp$, since the inverse inclusion follows from Proposition \ref{Prop_Unconditional Main Result}.

	To this end, fix $\lambda \in \sigma_{pt}\lc \tilde {S}_{\overline{d},\overline{p}} \rc$, and let $g_{\lambda} \in \mc{C}\lc \Z_{\overline{d}} \rc$ be a non-zero eigenvector associated with $\lambda$. 
	By Lemma \ref{Lem_Structure of non-autonomous Julia Set 1} and relation \eqref{Eq3_Projection of the Stochastic Operator}, $ g_{\lambda}\lc n \rc \; = \; c \cdot  v_{\lambda}(n)$ for all $n \in \Z_{+}$, 
	where $c= g\lc 0 \rc \neq 0$. 
	Without loss of generality, normalize $c =1$, so that $g\lc n \rc = v_{\lambda}\lc n \rc $ for all $n \in \Z_{+}$.
	
	Since the sequence $\lc q_{r-1} \rc_{r \ge 1}$ converges to $0$ in $\Z_{\overline{d}}$ as $r \to +\oo$, and since $g_{\lambda}\lc q_{r-1} \rc = \iota_{\lambda}\lc r \rc$ for each $r \ge 1$, it follows that
	$$\underset{r \to +\oo}{\lim} \iota_{\lambda}\lc r \rc \; = \; \lim _{r \to \infty} g_{\lambda}\lc q_{r-1} \rc \; = \; g_{\lambda}\lc 0 \rc \; = \;  v_{\lambda}\lc 0 \rc \; = \; 1 .$$ 
	
	The following claim completes the argument. 
	\begin{claim}\label{Claim_Structure of non-autonomous Julia Set 3} 
		Suppose that the sequence $\overline{d} = \lc d_{r} \rc_{r \ge 1} $ is bounded. If $ \underset{r \to +\oo}{\lim} \iota_{\lambda}\lc r \rc = 1$, then there exists $ r_{0} \ge 1$ such that $ \iota_{\lambda}\lc r \rc = 1 $ for every $r \ge r_{0}$.   
	\end{claim}

	Assuming the claim, there exists $r_{0} \ge 1$ such that $\iota_{\lambda}\lc r_{0} \rc = 1$, which implies $\tilde{f}_{r_{0}}\lc \lambda \rc = 1$. 
	Therefore, $\lambda \in \bigcup_{r=1}^{\infty} \tilde{f}_{r}^{-1}\lfp 1 \rfp$, completing the proof of the desired inclusion. 
	
	\paragraph{}
	It remains to prove Claim \ref{Claim_Structure of non-autonomous Julia Set 3}. Fix $\lambda \in  \C $ such that $\underset{r \to +\oo}{\lim} \iota_{\lambda}\lc r \rc =1 $, and suppose, for contradiction, that $\iota_{\lambda}\lc r \rc \neq 1$ for all $ r \ge 1$. 
	Using the recurrence relation \eqref{ps1}, one can express the difference $\iota_{\lambda}\lc r \rc -1 $ in terms of $\iota_{\lambda}\lc r- k \rc  -1 $ as  
	\begin{eqnarray}\label{ghbv}
		\iota_{\lambda}\lc r \rc -1 \; = \; \lc \iota_{\lambda}\lc r - k \rc -1\rc \cdot { z_{r-k +1} z_{r-k +2} \ldots z_{r} \over p_{r-k+1} p_{r-k +2} \ldots p_{r} }, 
	\end{eqnarray}
	for all $ 1 \leq k \leq  r-1$, where, for each $2 \le j \le r$, $z_{j} := \; 1+ \iota_{\lambda}(j-1)+ \cdots +\iota_{\lambda}(j-1)^{d_{j-1}-1}$. 
	
	Since the sequence $\overline{d} $ is bounded and $\underset{r \to +\oo}{\lim} \iota_{\lambda} \lc r \rc = 1$, relation \eqref{ghbv} implies that the sequence $ \lc \iota_{\lambda} \lc r \rc  \rc_{r \geq 1}$ must be unbounded--a contradiction. 
	Therefore, there exists $r_{0} \ge 1$ such that $\iota_{\lambda}\lc r_{0} \rc=1$, which proves the claim. Hence, the proof is complete. 
	
\end{proof}

\paragraph{}
To conclude the proof of Proposition \ref{Prop_Unconditional Main Result}, it remains to show that the spectrum of $\tilde{S}_{\overline{d}, \overline{p}}$ is contained in the non-autonomous filled Julia set $ \mc{E}_{\overline{d}, \overline{p}}$. For this purpose, we introduce the following shifted sequences of parameters: 
\begin{equation}\label{Eq5_Forward Sequences of Probabilities and Integers} 
	\overline{p}_{r} := \; \lc p_{r-1+j} \rc_{j=1}^\oo, \quad  \overline{d}_{r} := \; \lc d_{r-1+j} \rc_{j=1}^{\oo}, \quad \text{and define the space }\; \text{and }  \Z_{\overline{d}_{r}}. 
\end{equation}
Note that  $\overline{p}_1= \overline{p},\; \overline{d}_1= \overline{d}$ and $\Z_{\overline{d}_{1}} = \Z_{\overline{d}}$. 

Analogously to the operators $S_{\overline{d}, \overline{p}}$ and $\tilde{S}_{\overline{d}, \overline{p}}$, we define the following four families of operators:  
\begin{equation}\label{Eq_Inductive Operators}
	S_{r} := \;  S_{\overline{d}_{r}, \overline{p}_{r}}, \quad 
	R_{r} := \;  \frac{S_{r} - \lc 1 - p_{r} \rc I }{p_{r}}, \quad \tilde {S}_{r}:=  \; \tilde {S}_{\overline{d}_{r}, \overline{p}_{r}} \quad \text{and} \quad 
	\tilde{R}_{r} := \; \frac {\tilde {S}_{r} - \lc 1 - p_{r} \rc \tilde{I}_{r} } {p_{r}} 
\end{equation}
for all integers $r \geq 1$, where $I$ and $\tilde{I}_{r}$ denote the identity operators on $\ell^{\oo}\lc \Z_{+} \rc$ and $\mc{C}\lc \Z_{\bar{d}} \rc$, respectively. 
By construction, the norms of $\tilde{R}_{r}$ and $\tilde{S}_{r}$ satisfy the uniform bounds 
$$ || \tilde{S}_{r} || \le 1 \quad \text{and} \quad || \tilde{R}_{r} || \leq 1, \quad \text{for all }r \ge 1. $$

\paragraph{}

The next lemma encapsulates the inductive mechanism underlying the spectral inclusion \eqref{Eq_Theor_Unconditional Result}. It establishes a connection between the spectra of $\Tilde{R}_{r}^{d_{r}}$ and $\Tilde{S}_{r+1}$ through the family of polynomials $f_{r}$. A sketch of the proof is provided in the appendix.

\begin{lemma}\label{Lem1_Inductive_Spectra} 
	For all $r \geq 1$, the following equality holds: $$ f_{r}\lc \sigma \lc \tilde{S}_{r} \rc \rc  \; = \; \sigma \lc \tilde{R}_{r}^ {d_{r}} \rc \; = \; \sigma \lc  \tilde{S}_{r+1} \rc.$$
\end{lemma}

\begin{proof}[Proof of \eqref{Eq_Theor_Unconditional Result}] 
	
	The inclusion $\partial \mc{E}_{\overline{d}, \overline{p}} \sub \sigma\lc \tilde{S}_{\overline{d}, \overline{p}} \rc$ follows directly from the relations in \eqref{Eq_Theor_Unconditional Point Spectrum Inclusion}. 
	
	To prove the inclusion $\sigma\lc \tilde{S}_{\overline{d}, \overline{p}} \rc \sub \mc{E}_{\overline{d}, \overline{p}}$, fix $r \ge 1$. Observe first that $\sigma \lc \tilde{S}_{\overline{d}, \overline{p}} \rc =  \sigma \lc \tilde{S}_{1} \rc$. Then, by repeatedly applying Lemma \ref{Lem1_Inductive_Spectra}, one obtains $
	\tilde{f}_{r} \lc \sigma\lc \tilde{S}_{\overline{d}, \overline{p}} \rc \rc \; = \; \sigma\lc \tilde {S}_{r+1} \rc $. 
	Since $|| \tilde {S}_{r+1}  || \leq 1$, its spectrum lies within the closed unit disk $\overline{ B\lc 0, 1 \rc }$. Hence, for all $\lambda \in \sigma \lc \tilde{S}_{\overline{d}, \overline{p}} \rc$, one has $ \lav  \tilde{f}_{r} \lc  \lambda \rc \rav  \le 1 $. As the choice of $r \ge 1$ was arbitrary,  we conclude that $ \sigma\lc \tilde{S}_{\overline{d}, \overline{p}} \rc \sub \mc{E}_{\overline{d}, \overline{p}} $. 
	
\end{proof}

\subsection{Proof of Theorem \ref{Theor_Conditional Main Result} }\label{Section_Proof Conditional Results}

This section completes the proof of Theorem \ref{Theor_Conditional Main Result}. Recall that the right-hand side equality in \eqref{Eq__Theor_Conditional Main Result} was established in Proposition \ref{Prop_Unconditional Main Result}.   
It remains to prove the spectral equality stated in \eqref{Eq__Theor_Conditional Main Result}. 
We begin with the null recurrent case, that is, when $\prod_{r=1}^{+\oo} p_{r} = 0$.  
Since Proposition \ref{Prop_Unconditional Main Result} already gives the unconditional inclusion $ \sigma \lc \tilde{S}_{\overline{d}, \overline{p}} \rc \sub \mc{E}_{\overline{d}, \overline{p}} $, the problem reduces to proving the reverse inclusion $\mc{E}_{\overline{d}, \overline{p}} \sub \sigma \lc \tilde{S}_{\overline{d}, \overline{p}} \rc $.

\begin{proof}[Proof of \eqref{Eq__Theor_Conditional Main Result} (Case $\prod_{r=1}^{+\oo} p_{r} = 0$)]
	
	Assume that $\prod_{r=1}^{+\infty} p_{r} =0 $, and fix $\lambda \in \mc{E}_{\overline{d}, \overline{p}} $. The objective is to prove that $\lambda \in \sigma_{ap}\lc \tilde{S}_{\overline{d}} \rc$. 
	To this end, for each $t \ge 1$, define the continuous map  
	\begin{equation}\label{gfvcv}
		g_{\lambda, t}: \Z_{\overline{d}} \to \mathbb{C}, \quad		g_{\lambda, t}\lc x \rc : =  \;  \prod_{r=1}^{t} \lc \iota_{\lambda}\lc r \rc \rc^{a_{r}(x)} . 
	\end{equation}
	Two distinct cases must be considered, depending on whether the pointwise limit $\underset{t \to +\oo}{\lim} g_{\lambda, t}\lc x \rc$ defines a function in $\mc{C}\lc \Z_{\overline{d}} \rc$. 
	
	If the pointwise limit exists, it is given by 
	\begin{equation*}
		g_{\lambda}\lc x \rc  \; := \; \lim_{t \to +\oo} g_{\lambda,t}\lc x \rc \; = \;  \prod_{r=1}^{+\oo} \lc \iota_{\lambda}\lc r \rc \rc^{a_{r}(x)}  . 
	\end{equation*}
	In particular, if this limit defines a continuous function, then $\lambda$ is an eigenvalue; otherwise, $\lambda$ belongs to the approximate spectrum. As indicated by relation \eqref{Eq_Theor_Unconditional Point Spectrum Inclusion}, both scenarios indeed occur, depending on the position of $\lambda$ within $\mc{E}_{\overline{d}, \overline{p}}$.

	\paragraph{Case 1:} Assume that the pointwise limit of $ \lc  g_{\lambda, t} \rc_{t \geq 1}$ does not exist or defines a discontinuous function, i.e.,  $g_{\lambda} \notin \mc{C}\lc \Z_{\overline{d}} \rc$. 
	In this case, it suffices to show that  $ \ldav \lc \tilde{S}_{\overline{d}, \overline{p}} - \lambda I \rc g_{\lambda, t} \rdav_{\oo} \to 0$ as $ t \to +\oo $,   
	since $\ldav g_{\lambda, t} \rdav_{\oo} \ge 1$ for all $ t \ge 1$ (indeed, $g_{\lambda, t}\lc 0 \rc = 1$). 
	To this end, we employ the following identity, valid for all $s_{0} \ge 1$:  
	\begin{multline}\label{Eq5_Theor_Recurrent Case Inductive Identity}
		\lambda \cdot \prod_{r=1}^{s_{0}} \lc \iota_{\lambda}\lc r \rc \rc^{d_{r} -1 } \; = \; \lc 1 - p_{1} \rc \prod_{r=1}^{s_{0}} \lc \iota_{\lambda}\lc r \rc \rc^{d_{r} - 1} + \lc \prod_{r=1}^{s_{0} } p_{r} \rc \lc \iota_{\lambda}\lc s_{0}  \rc \rc^{d_{s_{0} } } \\
		+ \sum_{s=1}^{s_{0}  -1}\lc \lc \prod_{r=1}^{s} p_{r} \rc \rc \lc 1 - p_{s+1} \rc \prod_{r= s+1}^{s_{0} } \lc \iota_{\lambda}\lc r \rc \rc^{d_{r}-1}. 
	\end{multline} 
	This follows directly from \eqref{Eq3_Transition Operator}, \eqref{ps1}, and the relation $\lambda\, v_{\lambda}\lc q_{s_{0} - 1} \rc = \lc S_{\overline{d}, \overline{p}}\, \bm{v}_{\lambda}\rc \lc q_{s_{0} - 1}  \rc$.

	Let $ x\in \Z_{\overline{d}}$ be fixed. When the counter $s_{x}$ of $x$ satisfies $1 \le s_{x} \le t$, one obtains $\lc \tilde{S}_{\overline{d}, \overline{p}} - \lambda I \rc g_{\lambda, t} \lc x \rc = 0$. Indeed,  
	\begin{align*}
		\lc \tilde{S}_{\overline{d}, \overline{p}} \, g_{\lambda, t} \rc \lc x \rc \; & = \; \lc 1 - p_{1} \rc g_{\lambda, t} \lc x \rc + \lc \prod_{r=1}^{s_{x}} p_{r} \rc g_{\lambda, t} \lc x +1 \rc  + \sum_{s = 1}^{s_{x} -1} \lc \prod_{r=1}^{s} \lc p_{r} \rc \lc 1 - p_{s + 1} \rc  g_{\lambda, t} \lc T_{s}\lc x \rc \rc \rc . 
	\end{align*}
	This right-hand side simplifies to 
	\begin{align*}   
		& = \; \lc 1 - p_{1} \rc \prod_{r=1}^{s_{x}} \lc \iota_{\lambda} \lc r \rc \rc^{a_{r}\lc x \rc } \cdot \prod_{r= s_{x} + 1}^{t}  \lc \iota_{\lambda} \lc r \rc \rc^{a_{r}\lc x \rc }  + \lc \prod_{r=1}^{s_{x}} p_{r} \rc   \lc \iota_{\lambda} \lc s_{x} \rc \rc^{a_{s_{x}}\lc x  \rc  +1 }  \cdot \prod_{r= s_{x} + 1}^{t}  \lc \iota_{\lambda} \lc r \rc \rc^{a_{r}\lc x \rc } \\
		& \quad + \sum_{s = 1}^{s_{x} -1} \lc \lc \prod_{r = 1}^{s} p_{r} \rc \lc 1 - p_{ s  + 1} \rc \prod_{r = s + 1}^{s_{x}} \lc \iota_{\lambda} \lc r \rc \rc^{a_{r}\lc x  \rc } \cdot \prod_{r= s_{x} + 1}^{t}  \lc \iota_{\lambda} \lc r \rc \rc^{a_{r}\lc x  \rc }  \rc . \\ 
	\end{align*}    
	After inserting \eqref{Eq5_Theor_Recurrent Case Inductive Identity} (for $ s_{0} = s_{x}$) to the last sum,  one infers that
	\begin{center} $ \lc \tilde{S}_{\overline{d}, \overline{p}}g_{\lambda, t} \rc \lc x \rc = \lambda g_{\lambda, t}\lc x \rc$, hence  $ \ldav \lc \tilde{S}_{\overline{d}, \overline{p}} - \lambda I \rc g_{\lambda, t} \rdav_{\oo} = 0$. \end{center}

	When $1 \le t < s_{x}$, one has 
	
	\begin{align*}
		\lc \tilde{S}_{\overline{d}, \overline{p}}g_{\lambda, t} \rc \lc x \rc \; & = \;  \lc 1 - p_{1} \rc \prod_{r=1}^{t} \lc \iota_{\lambda} \lc r \rc \rc^{a_{r}\lc x \rc }  + \lc \prod_{r=1}^{s_{x}} p_{r} \rc + \sum_{s = t}^{s_{x} -1} \lc \prod_{r = 1}^{s} p_{r} \rc \lc 1 - p_{s+1} \rc   \\
		& \quad + \sum_{s = 1}^{t - 1} \lc  \lc \prod_{r = 1}^{s} p_{r} \rc \lc 1 - p_{s + 1} \rc \prod_{r= s +1}^{t} \lc \iota_{\lambda} \lc r \rc \rc^{a_{r}\lc x  \rc } \rc    .
	\end{align*}
	By adding and subtracting the term  $\lc \prod_{r = 1}^{t} p_{r} \rc \lc \iota_{\lambda}\lc t \rc \rc^{a_{t}\lc x \rc +1} $ and using \eqref{Eq5_Theor_Recurrent Case Inductive Identity} (for $s_{0} = t$), it follows that  
	$$ \tilde{S}_{\overline{d}, \overline{p}}  g_{\lambda, t } \; = \; \lambda g_{\lambda, t}\lc x \rc  +  \lc \prod_{r = 1}^{s_{x}} p_{r} \rc + \sum_{s = t}^{s_{x} -1} \lc \prod_{r = 1}^{s} p_{r} \rc \lc 1 - p_{s + 1} \rc  - \lc \prod_{r = 1}^{t} p_{r} \rc \lc \iota_{\lambda}\lc t \rc \rc^{a_{t}\lc x \rc +1} .$$
	Hence, 
	$$ \ldav \lc \tilde{S}_{\overline{d}, \overline{p}} - \lambda I \rc g_{\lambda, t} \rdav_{\oo} \; \le \; 3 \cdot \prod_{r = 1}^{t} p_{r} \; {\longrightarrow} \; 0  \quad \text{as} \quad t \to +\oo , $$
	since $\prod_{r=1}^{+\oo} p_{r} = 0$ by assumption. Therefore, $\lambda \in \sigma_{ap}\lc \tilde{S}_{\overline{d}, \overline{p}} \rc$.

	\paragraph{Case 2:} 
	If $g_{\lambda} \in \mc{C}\lc \Z_{\overline{d}} \rc$, then $\lambda$ is an eigenvalue. Indeed, applying \eqref{Eq5_Theor_Recurrent Case Inductive Identity} as before shows that $g_{\lambda}$ is an eigenfunction of $S_{\overline{d}, \overline{p}}$ associated with $\lambda$. Consequently, the proof is complete. 
	
\end{proof}

\paragraph{}

The proof of relation \eqref{Eq__Theor_Conditional Main Result} is now addressed in the case where the product of the probabilities satisfies $\prod_{r=1}^{+\oo} p_{r} \;> 0$, that is, in the transient case. 
Since the argument in this setting is more technical, it is instructive to begin with a simplified framework. To this end, consider a special, more tractable case in which the sequence $\lc p_{r} \rc_{r \ge 1}$ is identically equal to $1$, denoted by $\overline{1}:= \lc 1 \rc_{r \ge 1}$. 
In this situation, for each $r \ge 1$, the corresponding function reduces to $f_{r}\lc z \rc =  z^{d_{r}}$. The following claims can be verified: 
\begin{enumerate}
	\item The non-autonomous filled Julia set  $\mc{E}_{\overline{d}, \overline{1}}$ coincides with the closed unit disc $\overline{B\lc 0, 1 \rc}$. 
	\item The operator $\tilde{S}_{\overline{d}, \overline{1}}: \mc{C}\lc \Z_{\overline{d}} \rc \longrightarrow \mc{C}\lc \Z_{\overline{d}} \rc $ acts as a shift, that is, $ \lc \tilde{S}_{\overline{d}, \overline{1}} \rc g\lc x \rc = g\lc x+1 \rc$. 
	\item The spectrum $\sigma\lc \tilde{S}_{\overline{d}, \overline{1}}, \mc{C}\lc \Z_{\overline{d}} \rc \rc $ coincides with the unit circle $\SP^{1}:= \lfp \lambda \in \C :\; \lav \lambda \rav = 1 \rfp$. 
\end{enumerate}
The first two assertions follow immediately from the definitions. Concerning the third, Proposition \ref{Prop_Unconditional Main Result} implies that it suffices to verify the bijectivity of operator $\tilde{S}_{\overline{d}, \overline{1} } - \lambda \tilde{I}$ for every $\lambda \in \C$ with $\lav \lambda \rav < 1$. 
To prove surjectivity, let $q \in \mc{C}\lc \Z_{\overline{d}} \rc $ be arbitrary, and the function $g_{q}: \mc{C}\lc \Z_{\overline{d}}  \rc \to \mc{C}\lc \Z_{\overline{d}} \rc $ by $ g_{q}\lc x \rc :=  \sum_{n=0}^{+\oo} \lambda^{n} q\lc x - 1 - n \rc $. 
A direct computation shows that $\lc \tilde{S}_{\overline{d}, \overline{1}} - \lambda \tilde{I} \rc g_{q} = q$, demonstrating surjectivity. To prove injectivity, suppose that $g\in \mc{C}\lc \Z_{\overline{d}} \rc$ satisfies $\lc \tilde{S}_{\overline{d}, \overline{1}} - \lambda \tilde{I} \rc g  \equiv 0$. 
Then, for all $x \in  \Z_{\overline{d}}$ and $n \ge 1$, it follows that $g\lc x \rc = \lc 1/ \lambda \rc^{n} g\lc x +n \rc$. Since $\lav 1 / \lambda^{n} \rav \to +\oo$ as $ n \to +\oo$ and $g$ is bounded, this relation implies that $g \equiv 0$. Thus, the operator is injective, completing the argument.

\paragraph{}
Having clarified the special case, the analysis now turns to the general setting. The key tool is the following proposition, which asserts that if a complex number $\lambda$ belongs to the spectrum $\sigma\lc \tilde{S}_{\overline{d}, \overline{p}} \rc$, then the sequence $\lfp  \tilde{f}_{r}\lc \lambda \rc \rfp_{r \ge 1}$ converges in modulus to $1$ as $r \to +\oo$. 
The proof of this proposition is postponed until after the proof of relation \eqref{Eq__Theor_Conditional Main Result}.

\begin{proposition}\label{Prop_Transient Case Spectrum Limit}
	Let $\overline{d}= \lc d_{r} \rc_{r \ge 1}$ be a sequence of positive integers as defined \eqref{Eq1_Sequence of Integers and Probabilities}, and let $\overline{p} = \lc p_{r} \rc_{r \ge 1}$ be a sequence of probabilities satisfying $\prod_{r=1}^{+\oo} p_{r} \;> 0$. Then, for every $\lambda \in \sigma\lc \tilde{S}_{\overline{d}, \overline{p}} \rc$, the following limit holds: 
	\begin{equation}\label{Eq5_Prop_Transient Case Spectrum Limit}
		\lim_{r \to +\oo} \lav \tilde{f}_{r}\lc \lambda \rc \rav \; = \; 1 . 
	\end{equation}

\end{proposition}

\begin{proof}[Proof of \eqref{Eq__Theor_Conditional Main Result} (Case $\prod_{r=1}^{+\oo} p_{r} \;> 0$)] 
	
	Assume that the sequence $\overline{d}$ is bounded and that the infinite product $\prod_{r=1}^{+\infty} p_{r}  $ is strictly positive. 
	
	\medskip

	The objective is to show that the spectrum $\sigma\lc \tilde{S}_{\overline{d}, \overline{p} } \rc$ coincides with the boundary $ \partial \mc{E}_{\overline{d}, \overline{p}}$. 
	Once this equality is established, the standard spectral inclusions 
	(see, e.g., \cite[Lemma 1.2.13, p.17]{Davies_Linear Operators and Their Spectra})  
	\begin{equation*}
		\overline{\sigma_{pt}\lc \tilde{S}_{\overline{d}, \overline{p}}  \rc} \; \sub \; \sigma_{ap}\lc \tilde{S}_{\overline{d}, \overline{p}}  \rc \; \sub \; \sigma\lc \tilde{S}_{\overline{d}, \overline{p}}  \rc ,
	\end{equation*} 
	combined with the boundary relations in \eqref{Eq_Theor_Unconditional Point Spectrum Inclusion}, immediately imply that the spectrum coincides with the approximate point spectrum.

	To achieve this, it is sufficient to show that  
	\begin{equation}\label{Eq5_Prop5_Transient Case Limit Interior}
		\underset{r \to +\oo}{\lim} \tilde{f}_{r}\lc \lambda \rc  = 0 \quad \text{for all } \; \lambda \in inter\lc \mc{E}_{\overline{d}, \overline{p}} \rc .
	\end{equation} 
	Recall that relation \eqref{Eq_Theor_Unconditional Result} already provides the spectral inclusion  $\partial \mc{E}_{\overline{d}, \overline{p}} \sub \sigma\lc \tilde{S}_{\overline{d}, \overline{p}} \rc \sub \mc{E}_{\overline{d}, \overline{p}}$. Moreover, Proposition \ref{Prop_Transient Case Spectrum Limit} asserts that $\underset{r \to +\oo}{\lim}\lav \tilde{f}_{r}\lc \lambda \rc \rav =1$ for every $\lambda \in \sigma\lc \tilde{S}_{\overline{d}, \overline{p} } \rc$. 
	Thus, if \eqref{Eq5_Prop5_Transient Case Limit Interior} holds, then no point of the spectrum can lie in the interior of $\mc{E}_{\overline{d}, \overline{p}}$, and one concludes that $\sigma\lc \tilde{S}_{\overline{d}, \overline{p}} \rc = \partial \mc{E}_{\overline{d}, \overline{p}}$.  
	For clarity of exposition, the argument is divided into four parts.

	\paragraph{Part 1: Dichotomy of the Limit.} 
	To establish \eqref{Eq5_Prop5_Transient Case Limit Interior}, the first step is to prove a dichotomy describing the asymptotic behaviour of the sequence $\lc \tilde{f}_{r}\lc \lambda \rc \rc_{r \ge 1}$:   
	\begin{equation}\label{Eq5_Prop5_Transient Case Zero-One Law}
		\text{if } \; \lambda \in \mc{E}_{\overline{d}, \overline{p}} \quad \text{and} \quad \lim_{r \to +\oo} \lav \tilde{f}_{r} \lc \lambda \rc \rav \; \neq \; 1, \quad \text{then } \;  \lim_{r \to +\oo} \tilde{f}_{r} \lc \lambda \rc \;= \; 0. 
	\end{equation}
	Fix $\lambda \in \mc{E}_{\overline{d}, \overline{p}}$ such that $\lav \tilde{f}_{r}\lc \lambda \rc \rav < \rho_{0} \;< 1$ for infinitely many indices $ r \ge 1$, for some $0 \;< \rho_{0} \;< 1$. 
	To establish \eqref{Eq5_Prop5_Transient Case Zero-One Law}, it is sufficient to show that for every integer $j \ge 0$, there exists $r_{j} \ge 1$ such that 
	\begin{equation}\label{Eq5_Prop5_Transient Case Induction}
		\lav \tilde{f}_{r}\lc \lambda \rc \rav \; < \; \rho_{0}^{{\lc 3 / 2 \rc}^{j+1}} \quad \text{ for all } r \ge r_{j} +1 .     
	\end{equation}
	This will be shown by induction on $j$. 
	
	\textbf{Base step ($j=0$).} Since $\rho_{0} \in \lc 0, 1 \rc$, there exists $\delta_{0} \;> 0$ such that for all $0 \;< \delta \le \delta_{0}$, one has 
	\begin{equation}\label{Eq5_Prop5_Transient Case Contraction Argument}
		\lav {\lambda - \delta \over 1 - \delta } \rav^{d} \; < \;  \rho_{0}^{3\over 2} \quad \text{for all } d \ge 2, \quad \text{whenever } \; \lav \lambda \rav \; \le \; \rho_{0}  .
	\end{equation}
	Since $\prod_{r=1}^{+\oo} p_{r} \;> 0$, it follows that $ \underset{r \to +\oo}{\lim} p_{r} = 1$. Therefore, there exists $r_{0} \ge 1$ such that $\lav \tilde{f}_{r_{0}}\lc \lambda \rc \rav \le \rho_{0} $ and $ 1 - p_{r} \le \delta_{0}$ for all $r \ge r_{0} $. Then, by induction, 
	\begin{equation*}
		\lav \tilde{f}_{r+1}\lc \lambda \rc \rav \; = \; \lav { \tilde{f}_{r} \lc \lambda \rc - \lc 1 - p_{r} \rc \over p_{r} } \rav^{d_{r+1}} \; \underset{\eqref{Eq5_Prop5_Transient Case Contraction Argument}}{\le}  \rho_{0}^{3/2} \quad \text{for all } \; r \ge r_{0} +1   
	\end{equation*} 
	which establishes \eqref{Eq5_Prop5_Transient Case Induction} for $j=0$. 
	
	\textbf{Inductive step.} Assume that inequality \eqref{Eq5_Prop5_Transient Case Induction} holds for some $j \ge 0$, and set $\rho_{j} = \rho_{0}^{\lc 3/2 \rc^{j+1}}$. 
	Choose $\delta_{j+1} \;> 0$ so that \eqref{Eq5_Prop5_Transient Case Contraction Argument} holds with $\rho_{0}$ replaced by $\rho_{j}$. 
	Because $p_{r} \to 1$, there exists $r_{j+1} \ge 1$ such that $1 - p_{r} \le \delta_{j+1}$ for all $r \ge r_{j+1}$. Applying the reasoning as in the base case yields the desired inequality for $j+1$. 
	Thus, \eqref{Eq5_Prop5_Transient Case Induction} holds for all $j \ge 0$, which completes the proof of \eqref{Eq5_Prop5_Transient Case Zero-One Law}.

	\paragraph{Part 2: Dichotomy on the Connected Components.} 
	The goal is to show that, for any open and connected subset $W \sub \mc{E}_{\overline{d}, \overline{p}}$, the following dichotomy holds: 
	\begin{equation}\label{Eq5_Prop5_Transient Case Connected Sets Limit}
		\lim_{r \to +\oo} \tilde{f}_{r}\lc \lambda \rc \; = 
		\; 0 \quad \text{for all } \lambda \in W \quad \text{or} \lim_{r \to +\oo} \lav \tilde{f}_{r}\lc \lambda \rc \rav \; = \; 1 \quad \text{for all } \lambda \in W. 
	\end{equation} 
	Indeed, the sequence of analytic maps $\lfp \tilde{f}_{r} : W \to \C : \; r \ge 1 \rfp$ is uniformly bounded on $W$. Thus, by Montel's Theorem \cite[Chapter 3.2, Theorem 3.3, p. 225]{Stein_Shakarchi-Complex_Analysis}, this sequence forms a normal family. 
	Consequently, there exists a subsequence that converges uniformly on compact subsets of $W$ to a bounded analytic function  $\tilde{f}: W \to \C$. 
	
	By virtue of relation \eqref{Eq5_Prop5_Transient Case Zero-One Law}, the limit function satisfies $\tilde{f}\lc W \rc \sub \SP^{1} \cup \lfp 0 \rfp$. Since the image of a connected set under a continuous map is connected, it follows that $\tilde{f}\lc W \rc = \lfp 0 \rfp$ or $\tilde{f}\lc W \rc \sub \SP^{1}$. 
	
	Finally, applying once more relation \eqref{Eq5_Prop5_Transient Case Zero-One Law} yieds the desired dichotomy \eqref{Eq5_Prop5_Transient Case Connected Sets Limit}, thus completing this part of the proof.

	\paragraph{Part 3: Orbits of positive points.} 
	Let $W$ be the interior of a connected component of $\mc{E}_{\overline{d}, \overline{p}}$. The goal is to establish the existence of a point $\lambda_{w} \in W$ 
	\begin{equation}\label{Eq5_Prop5_Transient Case Existence of Contraction Point} 
		\underset{r \to +\oo}{\lim} \tilde{f}_{r}\lc \lambda_{w} \rc = 0   
	\end{equation}
	which, combined with the dichotomy in \eqref{Eq5_Prop5_Transient Case Connected Sets Limit}, implies that $\tilde{f}_{r}\lc \lambda \rc \to 0$ for every $\lambda \in W$. Hence, relation \eqref{Eq5_Prop5_Transient Case Limit Interior} follows. 
	
	To this end, let $\rho_{0} = 3/4$ and choose $\delta_{0} \le 1/4$ such that the contraction inequality \eqref{Eq5_Prop5_Transient Case Contraction Argument} holds for all $ 0 \;< \delta \le \delta_{0}$ and all $\lambda \in \C$ with $\lav \lambda \rav \le 3/4$. 
	Select $r_{0} \ge 1$ sufficiently large so that $1 - p_{r} \le \delta_{0}$ for every $ r \ge r_{0}$. Consequently, to prove \eqref{Eq5_Prop5_Transient Case Existence of Contraction Point}, it is enough to establish that  
	\begin{equation}\label{Eq5_Prop5_Transient Case Positive Point}
		\text{there exists } \; r_{w} \ge r_{0} \text{ and } \lambda_{w} \in W \text{ such that } \; \tilde{f}_{r_{w}}\lc \lambda_{w} \rc \in \lb 0 , 1 \rc.     
	\end{equation} 
	
	Indeed, if $0 \le \tilde{f}_{r_{w}}\lc \lambda_{w} \rc \le 3/4$, then inequality \eqref{Eq5_Prop5_Transient Case Contraction Argument} ensures that $\lav \tilde{f}_{r}\lc \lambda_{w} \rc \rav \le 3/4$ for all $ r \ge r_{w}$. 
	On the other hand, if $3/4 \;< \tilde{f}_{r_{w}}\lc \lambda_{w} \rc \;< 1$, a direct computation gives  
	$$ 0 \;< \; \tilde{f}_{r +1}\lc \lambda_{w} \rc \; = \; \lc {\tilde{f}_{r}\lc \lambda_{w} \rc - \lc 1 - p_{r} \rc \over p_{r} } \rc^{d_{r +1}} \; \le \; \tilde{f}_{r}\lc \lambda_{w} \rc^{d_{r +1}} , \quad \text{for all } r \ge r_{w} . $$
	Repeated application of this inequality yields a sequence $\lc \tilde{f}_{r}\lc \lambda_{w} \rc \rc_{r \ge r_{w}}$ of positive real numbers that eventually falls below $3/4$. 
	
	Therefore, in both cases one obtains $\lav \tilde{f}_{r} \lc \lambda_{w} \rc \rav \le 3/4$ for all sufficiently large $r$. 
	This shows that $\underset{r \to +\oo}{\lim} \lav \tilde{f}_{r}\lc \lambda_{w} \rc \rav \neq 1$, and, by \eqref{Eq5_Prop5_Transient Case Zero-One Law}, it follows that $\underset{r \to +\oo}{\lim} \tilde{f}_{r}\lc \lambda_{w} \rc = 0$. Hence, statement \eqref{Eq5_Prop5_Transient Case Positive Point} indeed implies \eqref{Eq5_Prop5_Transient Case Existence of Contraction Point}, as required.

	\paragraph{Part 4: Intersection with the real axis.}
	To complete the proof of \eqref{Eq5_Prop5_Transient Case Limit Interior}, it remains to establish relation \eqref{Eq5_Prop5_Transient Case Positive Point}. 
	By \eqref{Eq_Theor_Unconditional Point Spectrum Inclusion}, one has $\partial W \sub \overline{\bigcup_{r=1}^{+\oo} \tilde{f}_{r}^{-1} \lfp 1 \rfp  }$. Hence, one may, without loss of generality, choose $r_{0} \ge 1$ sufficiently large such that for all $r \ge r_{0}$ the image $\tilde{f}_{r}\lc W \rc$ is open and satisfies $ 1 \in \partial \tilde{f}_{r}\lc W \rc$ for all $  r \ge r_{0} -1 $. 
	
	For each $r \ge r_{0}$ and $\lambda \in \C\backslash\lfp 1 - p_{r} \rfp $, express 
	$${f}_{r}\lc \lambda \rc = \rho_{r}\lc \lambda \rc \cdot e^{\theta_{r}\lc \lambda \rc \cdot i} \quad \text{and} \quad h_{r}\lc \lambda \rc= \rho'_{r}\lc \lambda \rc \cdot e^{\theta'_{r}\lc \lambda \rc \cdot i }, $$
	where $\rho_{r}, \; \rho'_{r}: \C \backslash \lfp 1 - p_{r} \rfp  \to \R_{+} $ and 
	$$ \theta_{r}, \; \theta'_{r}: \C \backslash \lfp 1 - p_{r} \rfp\to  \R/ 2\pi \Z := \lfp \theta \pmod{2\pi} : \; \theta \in \R \rfp $$
	(i.e., $\theta$ is considered modulo $2\pi$). 
	At the point $\lambda = 1 - p_{r}$, set $\rho_{r}\lc 1 - p_{r} \rc = \rho'_{r}\lc 1 - p_{r} \rc = 0 $ and fix $\theta_{r}\lc 1 - p_{r} \rc = \theta_{r}'\lc 1 - p_{r} \rc = 0$ for definiteness. 
	Both $\theta_{r}$ and $\theta_{r}'$ are continuous on $\C \backslash \lfp  1- p_{r} \rfp$.

	\medskip 
	
	The goal is to show the existence of $r_{w} \ge r_{0}$ and $\lambda_{w} \in W$ such that $\theta_{r_{w} +1}\lc \tilde{f}_{r_{w}}\lc \lambda_{w} \rc \rc = 0 \pmod{2\pi}$, which implies $\tilde{f}_{r_{w} +1}\lc \lambda_{w} \rc \in \lb 0, 1 \rc$. 
	Define 
	$$ \lf \theta \rf_{2\pi} := \; \min\lfp \theta -  2\pi k \ge 0 : \; k \in \Z \rfp $$ 
	and assume, for contradiction, that $ \tilde{f}_{r}\lc \lambda \rc \not\in \lb 0, 1 \rc$ for all $r \ge r_{0}$ and $\lambda \in W$; otherwise, the desired conclusion already holds. 
	
	A simple geometric argument (see Figure \ref{Fig_Trigonometric Argument}) then yields the inequality    
	\begin{equation}\label{Eq5_Prop5_Transient Case Trigonometry 1}
		0 \; < \;  \lf \theta_{r}\lc \tilde{f}_{r-1}\lc \lambda \rc \rc \rf_{2\pi} \; \le \; \lf \theta'_{r+1}\lc \tilde{f}_{r}\lc \lambda \rc \rc  \rf_{2\pi} \; \le \; \pi, \quad \text{whenever }\; 0 \; < \; \lf \theta_{r}\lc \tilde{f}_{r-1}\lc \lambda \rc \rc \rf_{2\pi} \; \le  \; \pi . 
	\end{equation} 
	Moreover, since $ \tilde{f}_{r+1}\lc \lambda \rc = \lc h_{r+1}\lc \tilde{f}_{r}\lc \lambda \rc \rc\rc^{d_{r+1}}$, one has  
	\begin{equation}\label{Eq5_Prop5_Transient Case Trigonometry 2}
		\theta_{r+1}\lc \tilde{f}_{r}\lc \lambda \rc \rc \; = \; d_{r+1} \cdot \theta_{r+1}'\lc \tilde{f}_{r}\lc \lambda \rc \rc  \pmod{2\pi} , \quad \text{for all } \; r \ge 1 \;\text{ and }\; \lambda \in \C. 
	\end{equation}

	\begin{figure}
		\centering
		\includegraphics[width=16cm]{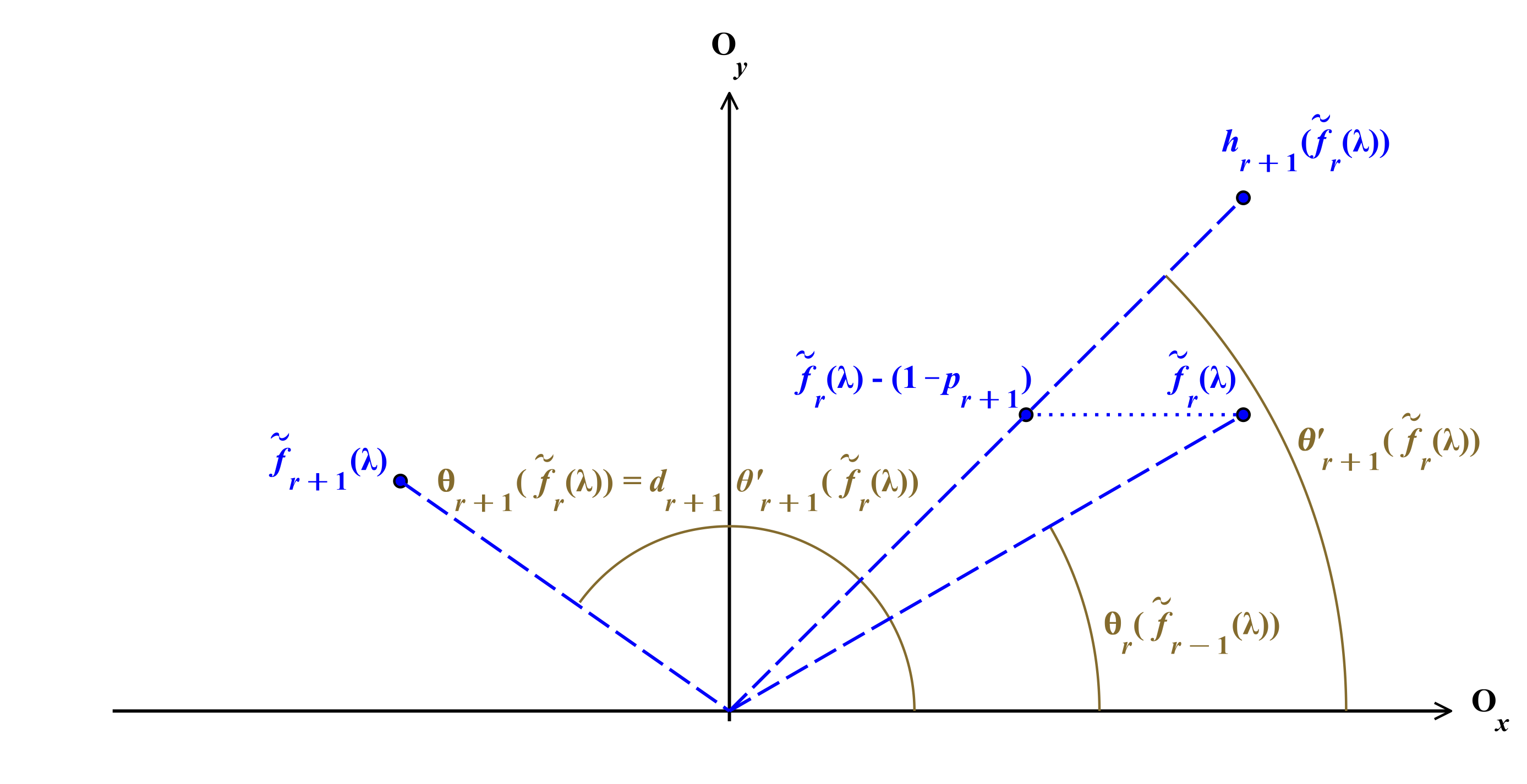}
		\caption{Illustration of the trigonometric relations \eqref{Eq5_Prop5_Transient Case Trigonometry 1} and \eqref{Eq5_Prop5_Transient Case Trigonometry 2}.} \label{Fig_Trigonometric Argument}
		
	\end{figure}
	
	Since $W$ is connected and $1 \in \partial \tilde{f}_{r_{0}}\lc W \rc$, there exists $\e_{0}\;> 0$ such that either $\lc 0 , \e_{0} \rc \sub A_{r_{0}}$ or $\lc 2\pi - \e_{0}, 2\pi \rc \sub A_{r_{0}}$.  In the latter case, replacing $\lambda$ with its conjugate $\overline{\lambda}$ shows that $\tilde{f}_{r}\lc \overline{\lambda} \rc =  \overline{\tilde{f}_{r}\lc \lambda \rc }$ and hence $$ \theta_{r}\lc \tilde{f}_{r-1}\lc \overline{\lambda} \rc \rc  =  - \theta_{r}\lc \tilde{f}_{r-1}\lc \lambda\rc  \rc \pmod{2\pi} .$$
	Therefore, one may assume without loss of generality that  
	\begin{equation}\label{Eq5_Prop_Assumption Conjugates}
		\lc 0, \e_{r_{0}} \rc \sub A_{r_{0}} \quad \text{for some }  \; \e_{r_{0}} \;> 0. 
	\end{equation} 
	
	Define $$A_{r} := \; \lfp \lf \theta_{r}\lc \tilde{f}_{r-1}\lc \lambda \rc \rc \rf_{2\pi} : \; \lambda \in W \rfp$$ and  
	$$ a_{r} :=  \; \sup\lfp  \lf \theta_{r}\lc \tilde{f}_{r-1}\lc \lambda \rc \rc \rf_{2\pi} :\; \lc 0 ,  \lf \theta_{r}\lc \tilde{f}_{r-1}\lc \lambda \rc \rc \rf_{2\pi} \rc \sub A_{r} \; \text{ and } \;  \lambda \in W \rfp $$
	for all $r \ge r_{0}$. 
	Relation \eqref{Eq5_Prop_Assumption Conjugates} and continuity of $\theta_{r}$, $\theta_{r}'$ ensure that $A_{r}$ contains an interval $\lc 0 , \e_{r} \rc$ for some $\e_{r} \;> 0$ and hence $a_{r} \;> 0$ with $ \lc 0, a_{r} \rc  \sub  A_{r}$ for all $ r \ge r_{0}$.

	By continuity and \eqref{Eq5_Prop5_Transient Case Trigonometry 1}--\eqref{Eq5_Prop5_Transient Case Trigonometry 2}, if $ a_{r} < \pi$, then $a_{r+1} \ge 2 a_{r}$. 
	Thus, there exists $r_{w} \ge r_{0}$ such that  $a_{r} \ge \pi$ for all $r \ge r_{w}$. 
	The continuity of $\theta_{r_{w}+1}':  \C \backslash \lfp 1 - p_{r} \rfp \to \R / 2\pi \Z$ then ensures the existence of $\lambda_{w} \in W$ satisfying $\theta'_{r_{w}+1} \lc \tilde{f}_{r_{w}} \lc \lambda_{w} \rc \rc = 2\pi/d_{r_{w}+1} $. 
	By \eqref{Eq5_Prop5_Transient Case Trigonometry 2}, this implies 
	$\theta_{r_{w} + 1}\lc \tilde{f}_{r_{w}}\lc  \lambda_{w} \rc \rc = 0 $, and hence $\tilde{f}_{r_{w} + 1}\lc \lambda_{w} \rc \in \lb 0, 1 \rc$.
	This establishes \eqref{Eq5_Prop5_Transient Case Positive Point}.

	Since $W \sub \mc{E}_{\overline{d}, \overline{p}}$ was chosen as the interior of an arbitrary connected component, the validity of \eqref{Eq5_Prop5_Transient Case Positive Point} implies \eqref{Eq5_Prop5_Transient Case Limit Interior}. 
	The proof is complete.

\end{proof}

\paragraph{}
To complete the proof of relation \eqref{Eq__Theor_Conditional Main Result} in the case where $\prod_{r=1}^{+\oo} p_{r} \;> 0$, it remains to establish Proposition \ref{Prop_Transient Case Spectrum Limit}. To this end, we make use of the following auxiliary result. 

\begin{lemma}\label{Lem5_Surjective Operators on Banach Spaces}
	Let $\lc X, \ldav \cdot \rdav \rc$ be a Banach space and $S: X \to X$ a continuous linear operator. Assume that there exist constants $0 < b < 1 < C$ such that, for every $q \in X$, there exists a point $g \in X$ satisfying 
	\begin{equation}\label{Eq5_Lem_Approximation Solutions}
		\ldav q - Sg \rdav \; \le \; b \cdot \ldav q \rdav \quad \text{and} \quad \ldav g \rdav \; \le \; C \ldav q \rdav.    
	\end{equation}
	Then the operator $S$ is surjective. 
	
\end{lemma}

\begin{proof}
	Fix $q_{0} \in X$ and define a sequence $\lc q_{j} \rc_{j \ge 0}$ recursively by $q_{j} = q_{j-1} - S g_{j-1}$,  $j \ge 1$, where each $g_{j-1} \in X$ is the vector provided by assumption \eqref{Eq5_Lem_Approximation Solutions} with $q = q_{j-1}$.  
	From \eqref{Eq5_Lem_Approximation Solutions}, one has $\ldav q_{j} \rdav \le c \ldav q_{j-1} \rdav$  for all $j \ge 1$, and hence $\ldav q_{j} \rdav \le b^{j} \ldav q_{0} \rdav$, and 
	$$ \ldav g_{j} \rdav \; \le \; C \ldav q_{j} \rdav \; \le \; C b^{j} \ldav q_{0} \rdav .$$ 
	Since $X$ is a Banach space, the series $g := \sum_{j=0}^{+\oo} g_{j}$ converges in $X$ and satisfies  
	$$ \ldav g \rdav \; \le \; C \sum_{j=0}^{+\oo} \ldav q_{j} \rdav \; \le \;  C \ldav q_{0} \rdav \sum_{j=0}^{+\oo} b^{j} \; \le \; {C \over 1 - b } \ldav q_{0} \rdav . $$
	Furthermore, since $q_{j} = q_{0} - S\lc\sum_{i=0}^{j-1} g_{i} \rc$ for all $j \ge 1$, taking the limit as $j \to +\oo$ yields  $S g = q_{0}$. As $q_{0} \in X$ was arbitrary, it follows that $S$ is surjective.  
	
\end{proof}

\medskip 

\begin{proof}[Proof of Proposition \ref{Prop_Transient Case Spectrum Limit}] 
	
	Let the sequence $\overline{d}$ of integers be a bounded and suppose     $\prod_{r=1}^{+\oo} p_{r} >  0$.  
	Recall that, for each $r \ge 1$, we define the shifted sequences $\overline{d}_{r} = \lc d_{r-1 +j} \rc_{j\ge 1}$ and $\overline{p}_{r} = \lc p_{r-1+j} \rc_{j \ge 1}$, as in \eqref{Eq5_Forward Sequences of Probabilities and Integers}.
	
	Fix $\lambda \in \sigma\lc \tilde{S}_{\overline{d}, \overline{p}} \rc$, and assume, without loss of generality, that 
	\begin{equation}\label{Eq5_Prop5_Point Spectrum Assumption}
		\tilde{f}_{r-1}\lc \lambda \rc  \in \sigma\lc \tilde{S}_{\overline{d}_{r}, \overline{p}_{r}} \rc \backslash \sigma_{pt}\lc \tilde{S}_{\overline{d}_{r}, \overline{p}_{r}} \rc \quad \text{for all }\; r \ge 1.
	\end{equation} 
	Indeed, by Lemma \ref{Lem1_Inductive_Spectra}, one has $\tilde{f}_{r-1}\lc \lambda \rc \in \sigma\lc \tilde{S}_{\overline{d}_{r}, \overline{p}_{r}} \rc $ for all $r \ge 1$. 
	If \eqref{Eq5_Prop5_Point Spectrum Assumption} fails for some $r_{0} \ge 1$, then, by relation \eqref{Eq_Theor_Unconditional Point Spectrum Inclusion}, we obtain $\tilde{f}_{r}\lc \lambda \rc = 1$ for all $r \ge r_{0}$, which immediately proves the assertion of the proposition. 
	
	\smallskip 
	The goal is to show, for all $r \ge 1$,  
	\begin{equation}\label{Eq5_Prop5_Spectrum Limit}
		2\prod_{s=r}^{+\oo}p_{s} - 1 \; \le \; \lav \tilde{f}_{r-1}\lc \lambda \rc \rav \; \le \; 1 .
	\end{equation}
	Taking the limit as $r \to +\oo$ in \eqref{Eq5_Prop_Transient Case Spectrum Limit} and using $ \prod_{s = r}^{+\oo} p_{s} \to 1$, we obtain the desired conclusion. Note that the upper bound in \eqref{Eq5_Prop5_Spectrum Limit} follows directly from Lemma \ref{Lem_Structure of non-autonomous Julia Set 1}. 
	
	\smallskip     
	To prove the lower bound, first note that \eqref{Eq5_Prop5_Spectrum Limit} holds trivially whenever $2 \prod_{s=r}^{+\oo} p_{s} - 1 \le 0$. Hence, fix $r \ge 1$ such that $2 \prod_{s=r}^{+\oo} p_{s} - 1 > 0$. 
	Without loss of generality, we may assume $r=1$, since by Lemma \ref{Lem1_Inductive_Spectra} and \eqref{Eq5_Prop5_Point Spectrum Assumption}, one can replace $\lambda$ by $\tilde{f}_{r-1}\lc \lambda \rc$ and $\tilde{S}_{\overline{d}, \overline{p}}$ by $\tilde{S}_{\overline{d}_{r}, \overline{p}_{r}}$. 
	Suppose, for contradiction, that $\lav \lambda \rav < 2 \prod_{s=1}^{+\oo} p_{s} -1$. 
	We claim that the operator $\lc \tilde{S}_{\overline{d}, \overline{p}} - \lambda \tilde{I} \rc$ is invertible. This would contradict the assumption that $\lambda \in \sigma\lc \tilde{S}_{\overline{d}, \overline{p}} \rc \backslash \sigma_{pt}\lc \tilde{S}_{\overline{d}, \overline{p}} \rc $. Thus we must have $\lav \lambda \rav \ge 2\prod_{s=1}^{+\oo} p_{s} -1 $, which establishes the left-hand inequality in \eqref{Eq5_Prop5_Spectrum Limit}.  
	
	\smallskip 
	To prove the claim, note that since $\lambda$ is not an eigenvalue, it suffices to show that the operator $ \tilde{S}_{\overline{d}, \overline{p}} - \lambda \tilde{I} $ is surjective. 
	This follows by applying Lemma \ref{Lem5_Surjective Operators on Banach Spaces}. Fix $q \in \mc{C}\lc \Z_{\overline{d}} \rc$ and define  
	$$ g\lc x \rc := \; { q\lc x -1 \rc \over \prod_{s=1}^{s_{x} } p_{s} } , \quad  b :=  {1 \over \prod_{s=1}^{+\oo} p_{s} } \cdot \lc \lav \lambda \rav + 1 - \prod_{s=1}^{+\oo} p_{s} \rc \;< \; 1 ,  \quad C :=  \prod_{s=1}^{+\oo} p_{s}^{-1} \;< \; 2 . $$
	It is readily verified that $g \in \mc{C}\lc \Z_{\overline{d}} \rc$ and $\ldav g \rdav_{\oo} \le C \ldav q \rdav_{\oo}$. Moreover, 
	$$ \ldav q - \lc \tilde{S}_{\overline{d}, \overline{p}} - \lambda \tilde{I} \rc g \rdav_{\oo} \; \le \;  b \cdot \ldav q \rdav_{\oo}, $$
	as shown below:   
	\begin{align*}
		\lav q\lc x \rc - \lc \tilde{S}_{\overline{d}, \overline{p}} - \lambda \tilde{I} \rc g \lc x \rc \rav \; & = \; \lav \lc 1-p_{1} - \lambda \rc g\lc x \rc + \sum_{t=1}^{s_{x} -1} \lc \prod_{s=1}^{t} p_{s} \lc 1 - p_{t+1} \rc g\lc T_{t}\lc x \rc \rc \rc \rav \\
		&\le \; {1 \over \prod_{s=1}^{s_{x}} p_{s} } \cdot \lc \lav \lambda \rav + 1 -p_{1} + \sum_{t=1}^{s_{x} -1 } \lc \lc 1 - p_{t+1} \rc \cdot \prod_{s=1}^{t} p_{s} \rc\rc \cdot \ldav q \rdav_{\oo} \\
		& \le {1 \over \prod_{s=1}^{s_{x}} p_{s} } \cdot \lc \lav \lambda \rav + 1 - \prod_{s=1}^{s_{x}} p_{s} \rc  \cdot \ldav q \rdav_{\oo}  \; < \;  b \cdot \ldav q \rdav_{\oo}. 
	\end{align*}
	Therefore, the hypotheses of Lemma \ref{Lem5_Surjective Operators on Banach Spaces} are satisfied, implying that $\tilde{S}_{\overline{d}, \overline{p}} - \lambda \tilde{I}$ is surjective. Hence, the claim is established, completing the proof of the proposition.

\end{proof}

\subsection{Proof of Corollary \ref{Cor_Dynamical Properties of Transition Operators}}\label{Sec_Cor_Dynamical Properties of Transition Operators}

Fix sequences $\overline{d}= \lc d_{r} \rc_{r \ge 1}$ and $\overline{p} = \lc p_{r} \rc_{r \ge 1}$.

\paragraph{Proof of Point $(1)$.} Assume that the spectrum $\sigma\lc \tilde{S}_{\overline{d}, \overline{p}}\rc$ is homeomorphic to a Cantor set. 
Observe first that the sequence $\overline{p}$ cannot be identically equal to $1$, since in that case the spectrum would coincide with the unit circle $\SP^{1}$. Let $r_{0} \ge 1$ be the smallest index such that $p_{r_{0}} \ne 1$. Then, by Theorem \ref{Theor_Conditional Main Result}, one has $ \tilde{f}_{r_{0}}^{-1}\lfp 1 \rfp \sub\sigma\lc \tilde{S}_{\overline{d}, \overline{p}} \rc$. It is immediate that $1 \in \sigma\lc \tilde{S}_{\overline{d}, \overline{p}} \rc$, and one checks that there also exists $\lambda \in \tilde{f}_{r_{0}}^{-1}\lfp 1 \rfp$ such that $\lav \lambda \rav < 1$. Consequently, there is no radius $r > 0$ for which every connected component of $\sigma\lc \tilde{S}_{\overline{d}, \overline{p}} \rc$ intersects the circle $\lfp z \in \C : \;  \lav z \rav = r \rfp$. By \cite[Theorem 1.4]{Bayart_Matheron-Dynamics_of_Linear_Operators}, it follows that the operator $\tilde{S}_{\overline{d}, \overline{p}}$ is not supercyclic.

\paragraph{Proof of Point $(2)$.} 
By assumption, one has $2\prod_{r=1}^{+\oo} p_{r} -1 > 0$. Fix any $ \kappa >  \lc 2 \prod_{r=1}^{+\oo} p_{r} -1 \rc^{-1}$. 
Since every connected component of the spectrum of a hypercyclic operator intersects the unit circle \cite[Theorem 1.18]{Bayart_Matheron-Dynamics_of_Linear_Operators}, it suffices to prove that $\sigma\lc \kappa \tilde{S}_{\overline{d}, \overline{p}} \rc \cap \lfp z \in \C : \; \lav z \rav = 1 \rfp = \emptyset$. To this end, it is enough to show that $\lav \lambda \rav > 2\prod_{r=1}^{+\oo} p_{r} -1 $, for all $ \lambda \in \sigma \lc \tilde{S}_{\overline{d}, \overline{p}} \rc $. This lower bound has already been established in the proof of Proposition \ref{Prop_Transient Case Spectrum Limit} (see inequality \eqref{Eq5_Prop5_Spectrum Limit}).

This completes the proof of the corollary. \hfill $  \square$

\begin{remark}
	\begin{enumerate}
		
		\item The conclusion of Point~(1) of Corollary 
		\ref{Cor_Dynamical Properties of Transition Operators}
		also holds for the operator $S_{\overline{d}, \overline{p}}$ acting on the
		spaces $c_{0}\lc \mathbb{Z}_{+} \rc$, $c \lc \mathbb{Z}_{+} \rc$, and
		$\ell^{q}\lc \mathbb{Z}_{+} \rc$ (for $1 \le q < \infty$).
		Indeed, in view of \cite[Theorem~4.1]{mv}, the same argument applies without
		modification.
		
		\item When $p_{r} = 1$ for all $r \ge 1$, the operator 
		$  S_{\overline{d}, \overline{1}} : \ell^{q}\lc \mathbb{Z}_{+} \rc 
		\longrightarrow \ell^{q}\lc \mathbb{Z}_{+} \rc $ reduces to the shift map and, by \cite[Theorem~4]{Bayart_Ruzsa-Difference_Sets_and_Frequently_Hypercyclic_Weighted_Shifts}, the operator $\kappa S_{\overline{d}, \overline{1}}$ is frequently hypercyclic for all $\kappa > 1$ and $1 \le q < +\oo$.
		Heuristic considerations suggest that the operator
		$S_{\overline{d}, \overline{p}}$ should be frequently hypercyclic on
		$\ell^{q}\!\lc \mathbb{Z}_{+} \rc$ (for $1 \le q < +\infty$) whenever $ \prod_{r=1}^{\infty} p_{r} > \max_{r \ge 1} \lfp 1 - p_{r} \rfp$, although a definitive answer to this question remains open.
		
		It should also be noted that, in contrast to the previous observation,
		the operator $\tilde{S}_{\overline{d}, \overline{1}}$ is not supercyclic
		\cite[Theorem~A]{BeltranMeneu_Jorda_MurilloArcilla-Supercyclicity_of_Weighted_Composition_Operators_on_Spaces_of_Continuous_Functions}.
		Taken together with Point-(2) of Corollary \ref{Cor_Dynamical Properties of Transition Operators}, we wonder if the operator $ \tilde{S}_{\overline{d}, \overline{p}} :
		\mathcal{C}\lc \mathbb{Z}_{\overline{d}} \rc
		\longrightarrow
		\mathcal{C}\lc \mathbb{Z}_{\overline{d}} \rc$ 
		is never supercylic or not.
	\end{enumerate}
\end{remark}

\section{Appendix}\label{Sec_Appendix}

This section provides the proof of Lemma \ref{Lem5_Boundary and Non-Normal Points} and a sketch of the proof of Lemma \ref{Lem1_Inductive_Spectra}.

\begin{proof}[Proof of Lemma \ref{Lem5_Boundary and Non-Normal Points}] It is first shown that  $\partial \mc{E}_{\overline{d},\overline{p}} \sub \lfp z \in \mathbb{C}:  \lfp \tilde{f}_{r} : \; r \geq 1 \rfp \; \text{ is not normal at }  z \rfp$. To this end, fix $z \in \C$.
	If $z \in \partial \mc{E}_{\overline{d},\overline{p}}$, then for every open neighborhood $U_{z} \sub \C$ containing $z$, there exist points $u, v \in V$ such that the sequence $\lc \tilde{f}_{r} \lc u \rc \rc_{r \geq 1}$ is bounded, while $ \underset{r \to +\oo}{\lim} \lav \tilde{f}_{r}\lc v \rc \rav = \infty$. Hence, the family $\lfp \tilde{f}_{r}: \; r \ge 1 \rfp $ fails to be normal at $z$, thus proving the inclusion. 
	
	For the reverse inclusion, suppose that $z \not\in \partial \mc{E}_{\overline{d},\overline{p}}$. 
	If $z \in \mathbb{C} \backslash  \mc{E}_{\overline{d}, \overline{p}}$, then, since $\mc{E}_{\overline{d}, \overline{p}}$ is compact, there exists a neighborhood $U_{z} \sub \C$ of $z$ such that $\underset{r \to +\oo}{\lim} \lav \tilde{f}_{r}\lc v \rc \rav = \infty$ for all $v \in U_{z}$, by Lemma \ref{Lem_Structure of non-autonomous Julia Set 1}. 
	Thus, the family $ \lfp  \tilde{f}_{r} : \; r \geq 1 \rfp$ is normal at $z$.
	If instead $z \in  inter\lc E_{\overline{d}, \overline{p}} \rc$, then there exists an open neighborhood $U_{z} \sub \C$ of $z$ such that $ \lav \tilde{f}_{r} \lc v \rc \rav \le 1$ for all $v \in U_{z}$ and all $r \ge 1$.
	Therefore, the family $ \lfp  \tilde{f}_{r} : \; r \geq 1 \rfp $  is locally uniformly bounded in $U_{z}$, and hence normal at $z$ by Montel's Theorem \cite[Chapter 3.2, Theorem 3.3, p. 225]{Stein_Shakarchi-Complex_Analysis}.

	\paragraph{} 
	The proof of the second part of the lemma now follows. Let $U \sub \C$ be an open set such that $U \cap \partial \mc{E}_{\overline{d}, \overline{p}} \neq  \emptyset$. The aim is to prove that there exists $r \ge 1$ for which $\tilde{f}_{r}^{-1}\lfp 1 \rfp \cap U \neq \emptyset$. Two cases are considered. 
	
	\paragraph{Case 1.} Assume that there exists $d \ge 2$ such that $d$ divides $d_{r}$ for infinitely many $r \ge 1$. Thus, there exists a subsequence $ \lc r_{k} \rc_{k=1}^{+\oo}$ and an integer $d \ge 2$ such that $d_{r_{k}} = d \cdot d'_{r_{k}}$ for some integers $d'_{r_{k}} \ge 1$. 
	By Lemma \ref{Lem_Structure of non-autonomous Julia Set 1}, and in particular by relation \eqref{Eq_Divergent Sequences}, the family $\lfp \tilde{f}_{r} : \; r \ge 1 \rfp$ is normal at $z \in \C$ if and only if the subfamily $\lfp \tilde{f}_{r_{k}} : \; k \ge 1 \rfp$ is normal at $z$.
	Equivalently, since $\tilde{f}_{r_{k}} = \lc h_{r_{k}}\circ \tilde{f}_{r_{k}-1} \rc^{d \cdot d'_{r_{k}}}$, where the maps $h_{r_{k}}: \C \mapsto \C$ for $k \ge 1$ are defined in \eqref{ps2}, this happens if and only if the family $ \lfp h_{r_{k}} \circ \tilde{f}_{r_{k} -1} : \; k \ge 1 \rfp$ is normal at $z \in \C$. 
	
	Fix $w \in \C \backslash \lfp 0 \rfp$, and write its polar form as $w =  \rho_{w} e^{2 \pi i \theta_{w}}$ with $\rho_{w} \;> 0$ and $\theta_{w} \in \R$. For each $1 \le j \le d$, set $w_{j} = \rho_{w}^{1 / d} \cdot e^{2\pi i \cdot\lc {\theta_{w} \over d} + {j \over d} \rc }$. 
	Then $w = w_{j}^{d}$ for each $j$, so that the $w_{j}$ are the $d$-th roots of $w$. 
	Since $\tilde{f}_{r_{k}} = \lc h_{r_{k}}\circ \tilde{f}_{r_{k}-1} \rc^{d \cdot d'_{r_{k}}}$,  one has $w \in \tilde{f}_{r_{k}}\lc U \rc$ whenever $w_{j} \in h_{r_{k}}\circ \tilde{f}_{r_{k}-1}\lc U \rc$ for some $1\le j \le d$. 
	We claim that for some $1 \le j \le d$ and $k \ge 1$, the inclusion $w_{j} \in h_{r_{k}}\circ \tilde{f}_{r_{k}-1}\lc U \rc$ holds. Indeed, by Montel's Normality Theorem \cite[Theorem 9.15, p. 340]{Markushevich-Theory_of_Functions_of_Complex_Variable_III}, if $U$ contains a non-normal point, then the union $ \bigcup_{r \ge 1}  h_{r_{k}}\circ \tilde{f}_{r_{k}-1}\lc U \rc = \C \backslash X$, where $X$ is either empty or consists of a single point. 
	Since $U$ intersects $\partial \mc{E}_{\overline{d}, \overline{p}}$, it contains non-normal points (by the first part of the proof). Therefore, as $d\ge 2$, at least one of the $w_{j}$ lies in $ \bigcup_{r \ge 1}  h_{r_{k}}\circ \tilde{f}_{r_{k}-1}\lc U \rc$, establishing the claim. 
	Consequently, if the union $\tilde{f}_{r}\lc U \rc$ omits a point, that point must be $0$. Therefore $1 \in \tilde{f}_{r}\lc U \rc$ for some $r \ge 1$, completing the argument in this case. 
	
	\paragraph{Case 2.} 
	Assume that for every $d \ge 2 $ only finitely many $d_{r}$ are divisible by $d$. Fix $d =2$. Using an argument analogous to the one in Case 1, it follows that $\C \backslash \lfp 0 \rfp \; \sub \; \bigcup_{r = 1}^{+\oo} \lc h_{r} \circ \tilde{f}_{r-1} \rc^{d_{r} -1 }\lc U \rc $. 
	In particular, there exists $r \ge 1$ such that $1 \in \lc h_{r} \circ \tilde{f}_{r-1} \rc^{d_{r} -1 }\lc U \rc$. Since $h_{r}\lc 1 \rc = 1$, it follows that $1 \in \lc h_{r} \circ \tilde{f}_{r-1} \rc^{d_{r} }\lc U \rc =  \tilde{f}_{r}\lc U \rc$. Thus $U$ contains a preimage of $1$ under $\tilde{f}_{r}$, completing the argument in this case. 
	
	The proof of the lemma is therefore complete. 
	
\end{proof}

\paragraph{}

\paragraph{Sketch of Proof of Lemma \ref{Lem1_Inductive_Spectra}.} 
By the Spectral Mapping Theorem \cite[Theorem 1.2.18, p. 18]{Davies_Linear Operators and Their Spectra}, one has 
$ f_{r} \lc \sigma\lc \tilde{S}_{r} \rc \rc =  \sigma\lc \Tilde{R}_{r}^{d_{r}} \rc  $, where the operators $\tilde{S}_{r}$ and $\tilde{R}_{r}$ are defined in \eqref{Eq_Inductive Operators}. Hence, it remains to prove that $\sigma \lc \tilde{R}_{r}^ {d_{r}} \rc  =  \sigma \lc \tilde{S}_{r+1} \rc $. 

For each $r \ge 1$ and $0\le k \le d_{r}-1$, consider the auxiliary maps 
$$ \tilde{\Pi}_{k, r}: \mc{C} \lc \Z_{\overline{d}_{r+1}} \rc \longrightarrow  \mc{C} \lc \Z_{\overline{d}_{r}} \rc,  \quad \tilde{F}_{k, r}: \mc{C}\lc \Z_{\overline{d}_{r}} \rc \longrightarrow \mc{C}\lc \Z_{\overline{d}_{r+1}} \rc, $$ 
defined for $g \in \mc{C} \lc \Z_{\overline{d}_{r+1}} \rc $ and $h \in \mc{C} \lc \Z_{\overline{d}_{r}} \rc$ by  
\begin{equation*}
	\lc \tilde{\Pi}_{k,r} \lc g \rc \rc \lc x \rc :=   \begin{cases}
		g\lc a_{r+1} a_{r+2} \cdots \rc  &  , \; \text{ if } a_{r} = k \\ 
		0 &  , \; \text{ if } a_{r} \neq k   \\  
	\end{cases},  \quad \text{for all } x = a_{r}a_{r+1}\cdots \in \Z_{\overline{d}_{r}}   
\end{equation*} 
and 
$$ \lc \tilde{F}_{k, r}\lc h \rc \rc \lc x \rc : = \;  h \lc k a_{r+1} a_{r+2} \ldots \rc, \quad \text{for all } x= a_{r+1} a_{r+2} \ldots \in \Z_{\overline{d}_{r+1}} . $$

With this notation, the identity operator $\Tilde{I}_{r}: \mc{C}\lc \Z_{\overline{d}_{r}} \rc \longrightarrow \mc{C}\lc \Z_{\overline{d}_{r}} \rc$ decomposes as
\begin{equation}\label{Eq_Decomposition of the Identity Operator}
	\Tilde{I}_{r} \; = \; \sum_{k=0}^{d_{r} -1} \Tilde{\Pi}_{k,r} \circ \Tilde{F}_{k,r} .
\end{equation}

The operator $\tilde{R}_{r}^{d_{r}}$ can now be expressed in terms of $\tilde{S}_{r+1}$ using the auxiliary maps:    
$$ \tilde{R}_{r}^ {d_{r} } \; = \; \sum_{k=0}^{d_{r} -1} \tilde{\Pi}_{k,r} \circ \tilde{S} _{r+1} \circ \tilde{F}_{k, r}. $$ 
Combining this with \eqref{Eq_Decomposition of the Identity Operator} gives, for all $\lambda \in \C$,  
\begin{equation}\label{Eq_RrSr_Relation} 
	\lc \tilde{R}_{r}^ {d_{r}} - \lambda \tilde{I}_{r} \rc \; = \;  \sum_{k=0}^{d_{r}-1} \tilde{\Pi}_{k,r} \circ \lc  \tilde{ S}_{r+1} - \lambda  \tilde{I}_{r+1} \rc \circ \tilde{F}_{k,r}. 
\end{equation}
From \eqref{Eq_RrSr_Relation}, it follows that, for every $\lambda \in \C$, the operator
$ \Tilde{R}_{r}^{d_{r}} - \lambda \Tilde{I}_{r} $ is bijective if and only if
$ \Tilde{S}_{r+1} - \lambda \Tilde{I}_{r+1} $ is bijective.
This completes the proof of Lemma~\ref{Lem1_Inductive_Spectra}.

\begin{remark} 
	In \cite{msv}, Lemma \ref{Lem1_Inductive_Spectra} was established in the special case where the sequence $\overline{d}$ is constant, equal to a fixed integer $d \ge 2$, and for the operator $S_{\overline{d},\overline{p}}: \ell^{\oo}\lc \Z_{+} \rc   \longrightarrow   \ell^{\oo}\lc \Z_{+} \rc$  in place of $\tilde{S}_{\overline{d}, \overline{p}}: \mc{C}\lc \Z_{\overline{d}} \rc \longrightarrow \mc{C}\lc \Z_{\overline{d}} \rc$.
	The proof given above is an adaptation of the argument in \cite[Lemma~3.5]{msv}. 
\end{remark}

\paragraph{}

\thanks{Ali Messaoudi was partially supported by CNPq grant 310784/2021-2 and by Fapesp projects 2025/16695-5 and  2024/04685-2. 
	Ioannis Tsokanos was supported by Fapesp grant 2024/10135-5. Glauco Valle was supported by CNPq grants 307938/2022-0 and 403423/2023-6 and by FAPERJ grant E-26/200.442/2023.}

\vfill

\noindent \hrulefill

\noindent D. Caprio -- UNESP - Departamento de matemática, FEIS - Faculdade de Engenharia - Câmpus de Ilha Solteira, Avenida Brasil, 56 - Centro - Ilha Solteira, SP - CEP 15385-000, SP, Brasil. e-mail: \texttt{danilo.caprio@unesp.br}

\noindent A. Messaoudi -- UNESP - Departamento de matemática, Instituto de Biociências Letras e Ciências Exatas, Rua Cristóvão Colombo, 2265, Jardim Nazareth, 15054-000 - São José do Rio Preto, SP, Brasil. e-mail: \texttt{ali.messaoudi@unesp.br}

\noindent I. Tsokanos -- UNESP - Departamento de matemática, Instituto de Biociências Letras e Ciências Exatas. e-mail: \texttt{ioannis.tsokanos@unesp.br}

\noindent G. Valle -- UFRJ - Departamento de métodos estatísticos do Instituto de Matemática. Caixa Postal 68530, 21945-970, Rio de Janeiro, Brasil. e-mail: \texttt{glauco.valle@im.ufrj.br}

\pagebreak

\begin{figure}[!h]
	\centering
	a)	\includegraphics[scale=0.048]{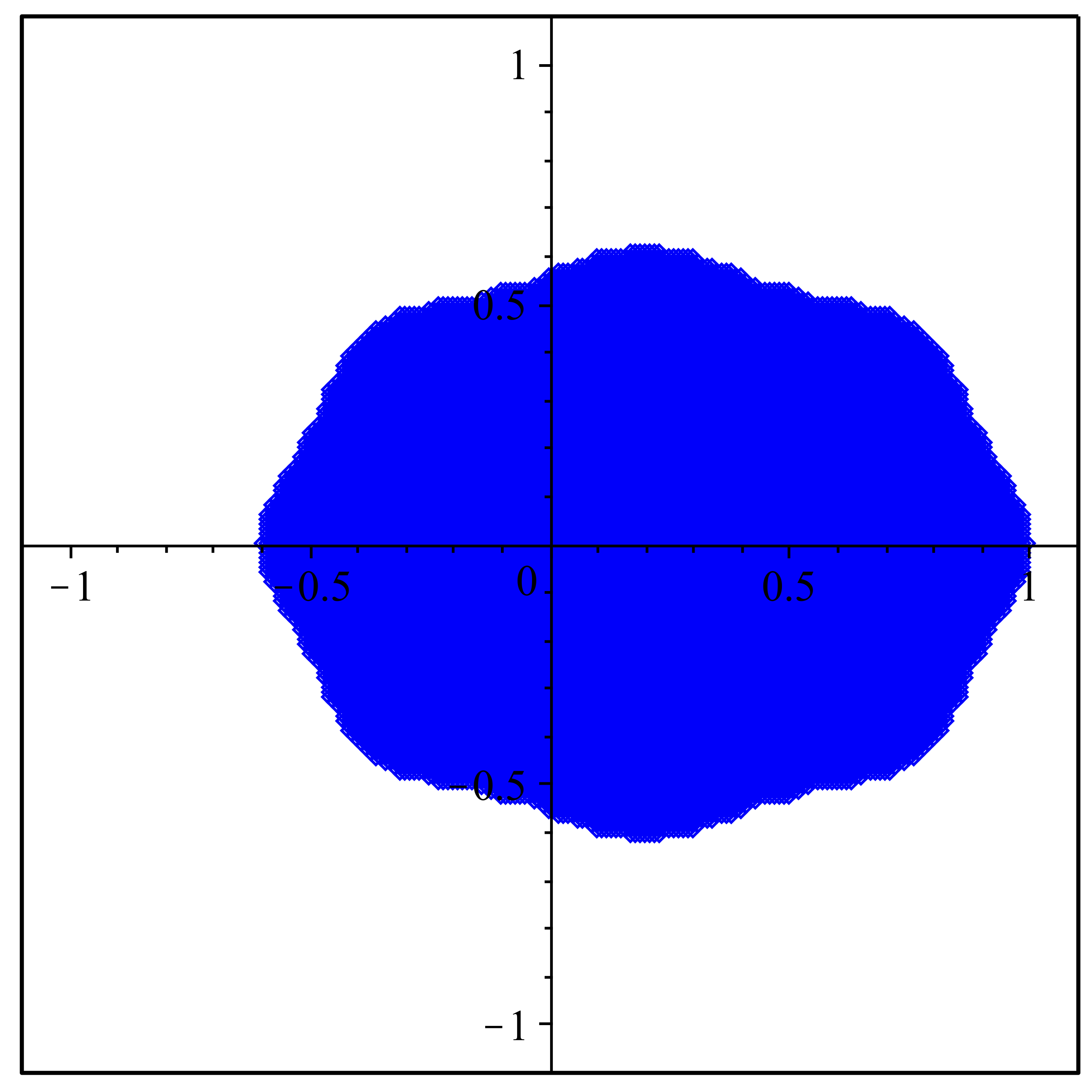}
	b)	\includegraphics[scale=0.048]{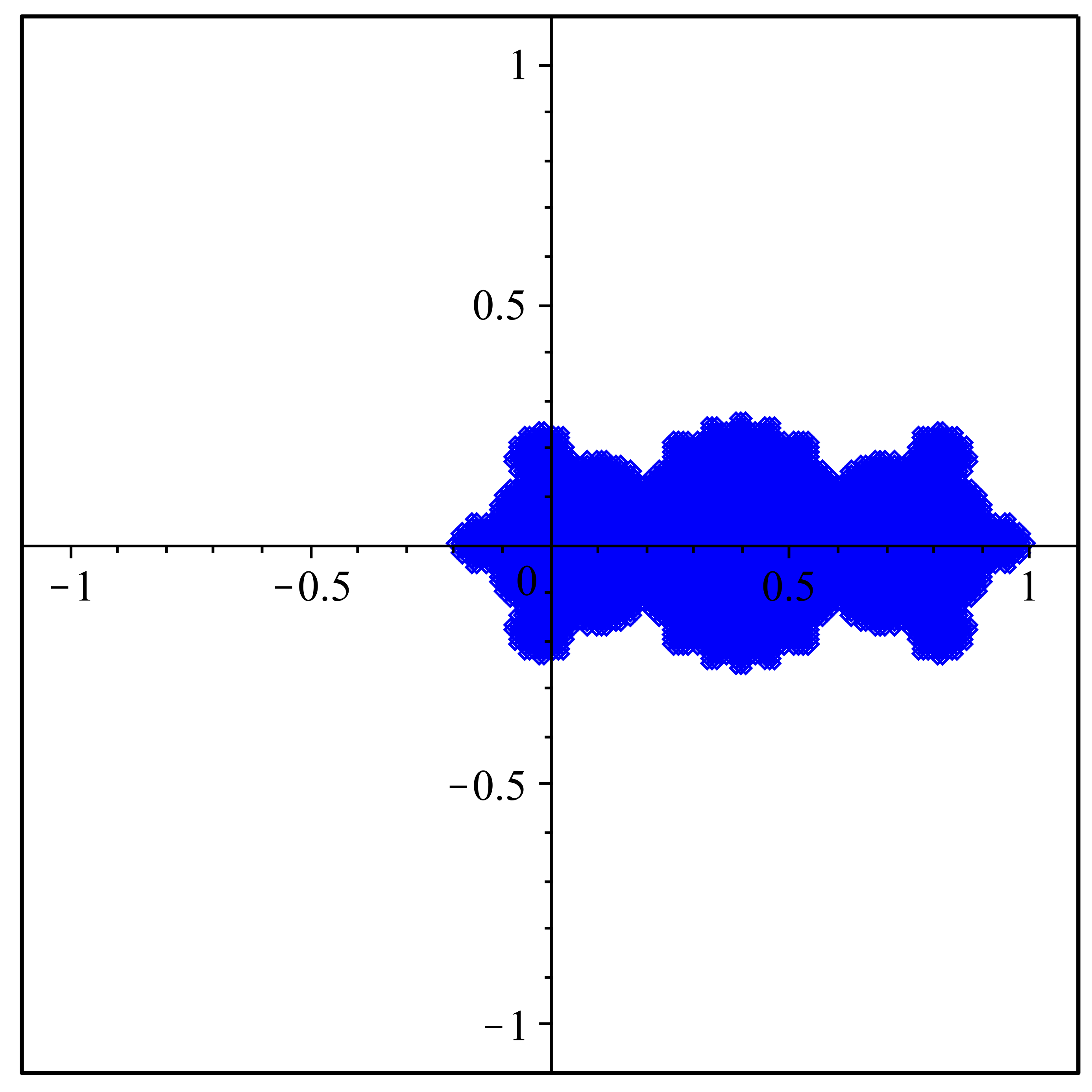}
	c)	\includegraphics[scale=0.048]{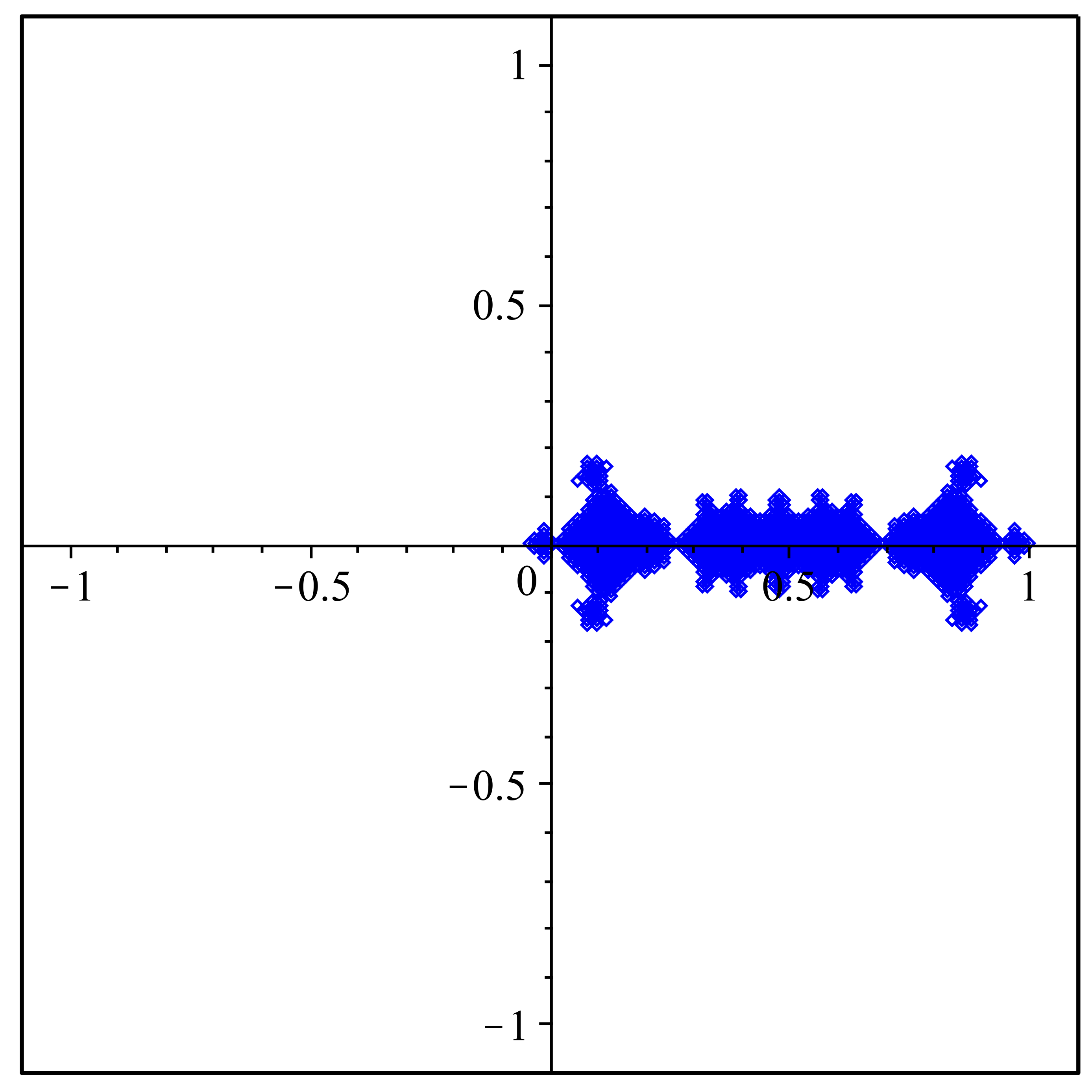}
	\caption*{Illustrations of the sets $\mc{E}_{\overline{d}, \overline{p}}$ in the case $d_{r} = 2r$ for all $r \ge 1$.  
		Displayed are the sets corresponding to constant probabilities  
		(a) $p_{r} = 0.8$, (b) $p_{r} = 0.6$, and (c) $p_{r} = 0.52$.}
\end{figure}

\begin{figure}[!h]
	\centering
	a)	\includegraphics[scale=0.048]{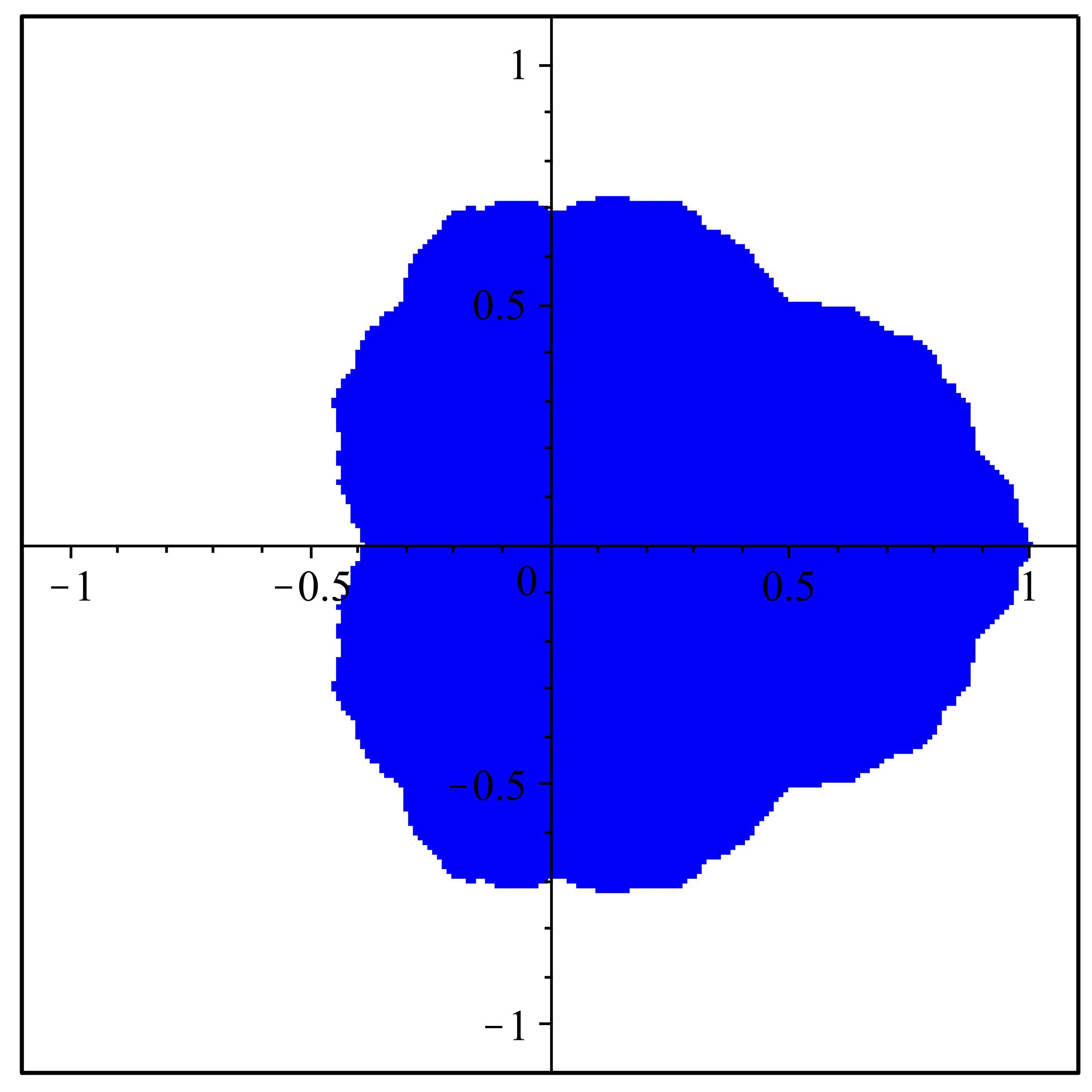}
	b)	\includegraphics[scale=0.048]{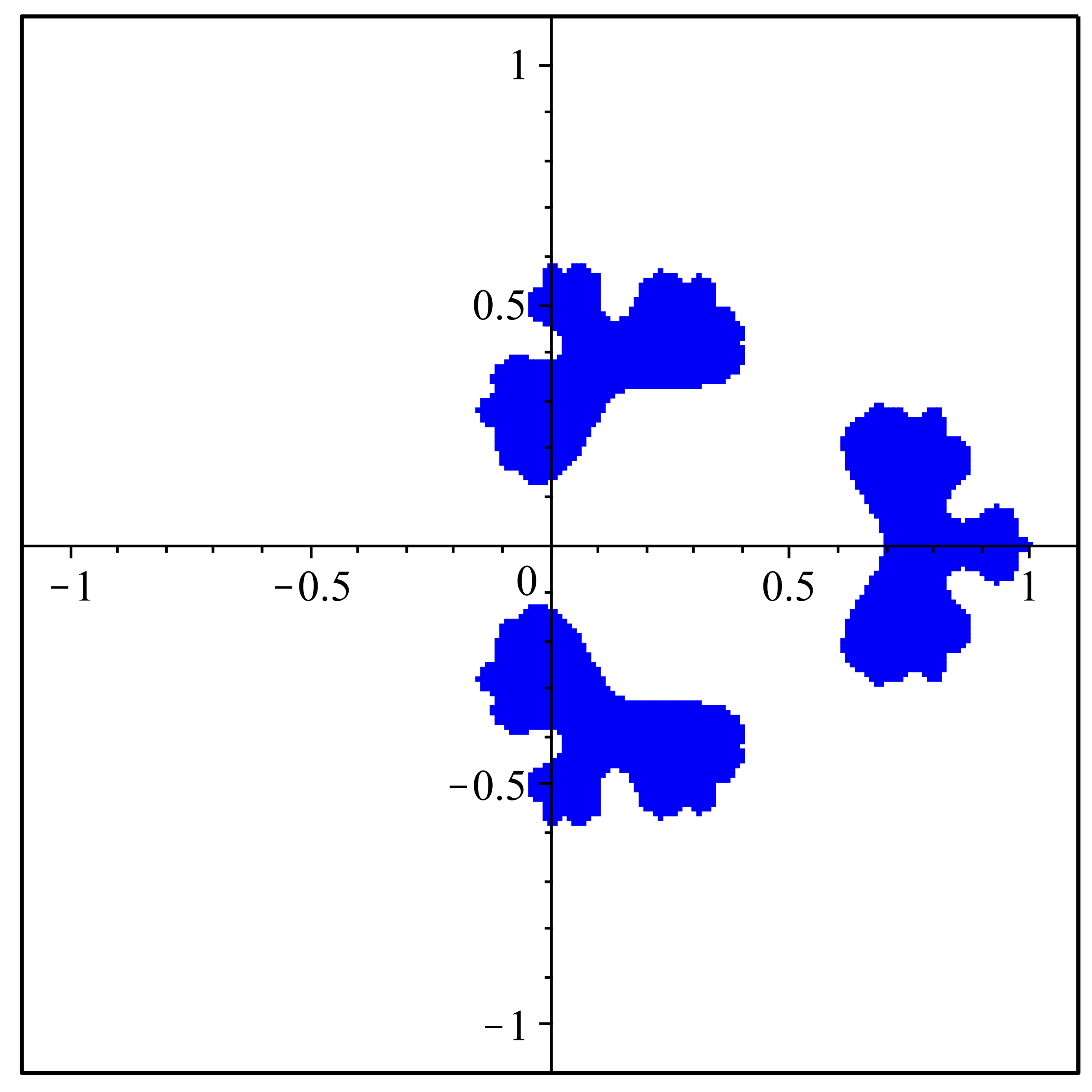}
	c)	\includegraphics[scale=0.048]{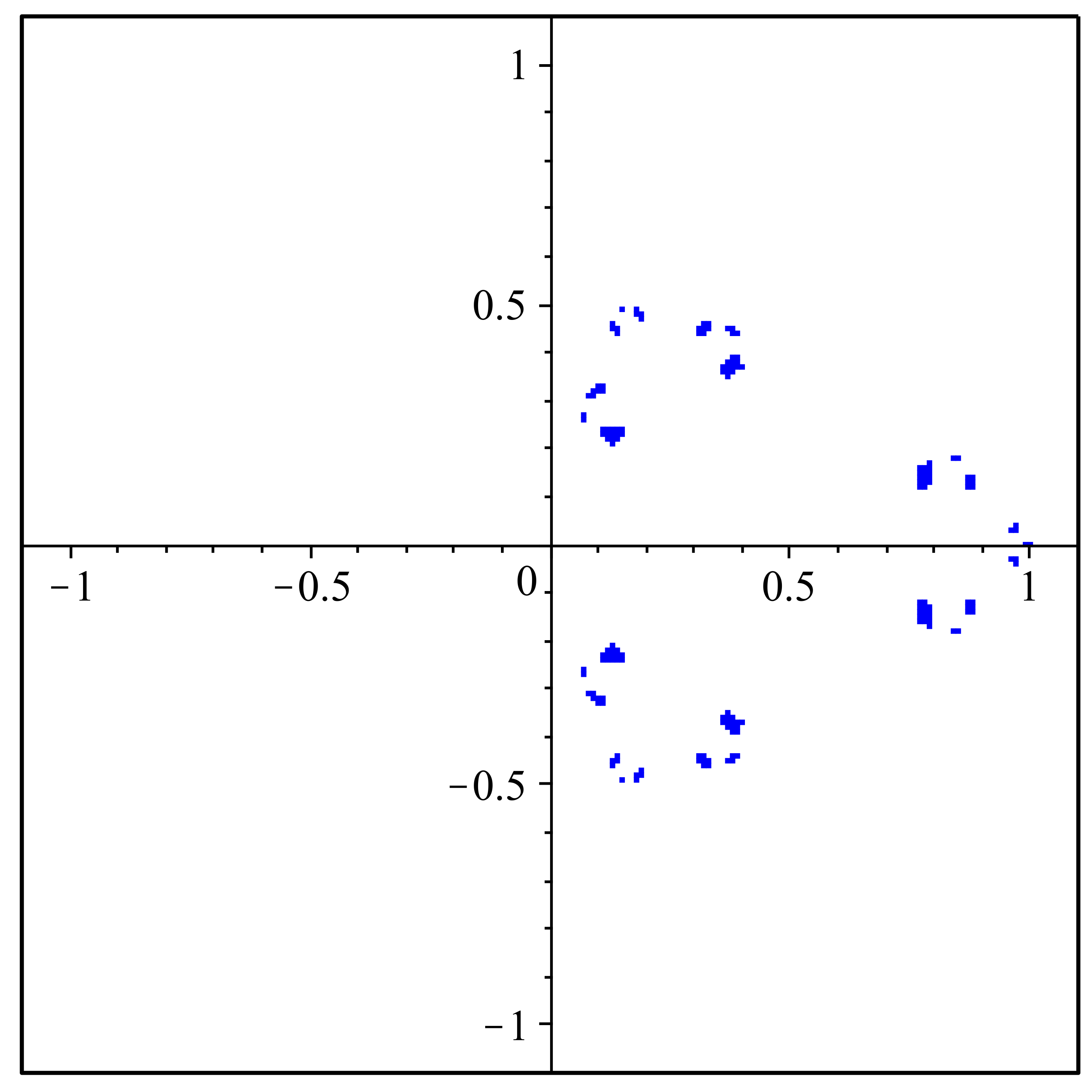}
	\caption*{The sets $\mc{E}_{\overline{d}, \overline{p}}$ for the ternary (base-$3$) numeration system, i.e., $d_{r} = 3$ for all $r \ge 1$.
		Panels (a), (b), and (c) correspond to the choices $p_{i} = 0.8$, $0.67$, and $0.57$ for $1\leq i \leq 5$, respectively, with $p_{r} = 1$ for all $r \ge 6$.}
\end{figure}

\end{document}